\documentclass[11pt]{amsart}
\usepackage{amsmath, amssymb}

\newtheorem{theorem}{Theorem}[section]
\newtheorem{corollary}[theorem]{Corollary}
\newtheorem{lemma}[theorem]{Lemma}
\newtheorem{prop}[theorem]{Proposition}
\newtheorem*{thmA}{Theorem A}

\theoremstyle{definition}
\newtheorem{example}[theorem]{Example}
\newtheorem{question}[theorem]{Question}
\newtheorem{examples}[theorem]{Examples}

\newtheorem{rem}[theorem]{Remark}

\newtheorem{definition}[theorem]{Definition}

\newcommand{\dist}{\operatorname{dist}}
\newcommand{\spann}{\operatorname{span}}
\newcommand{\C}{\mathbb{C}}
\newcommand{\Y}{Y}

\newcommand{\N}{\mathbb{N}}
\newcommand{\R}{\mathbb{R}}

\newcommand{\Z}{\mathbb{Z}}
\newcommand{\A}{\mathbb{A}}
\newcommand{\Der}{\operatorname{Der}}

\newcommand{\pt}{\partial}

\newcommand{\dbar}{\bar{\partial}}

\newcommand{\cB}{\mathcal{B}}

\newcommand{\cD}{\mathcal{D}}
\newcommand{\cH}{\mathcal{H}}
\newcommand{\cJ}{\mathcal{J}}
\newcommand{\cG}{\mathcal{G}}
\newcommand{\cQ}{\mathcal{Q}}
\newcommand{\cO}{\mathcal {O}}

\newcommand{\tpsi}{\tilde\psi}

\newcommand{\Hdot}{H_0}
\newcommand{\h}{H}
\newcommand{\hdot}{H_0}

\newcommand{\aaaa}{a}
\newcommand{\wh}{\widehat}
\newcommand{\uu}{u}
\newcommand{\vv}{v}
\newcommand{\GH}{g}
\newcommand{\toe}{\tau}

\newcommand{\beqn}{\begin{equation}}
\newcommand{\neqn}{\end{equation}}

\newcommand{\Carl}{\operatorname{Carleson}}

\renewcommand{\Re}{\operatorname{Re}}
\renewcommand{\Im}{\operatorname{Im}}

\newcommand{\Mod}{\operatorname{Mod}}

\newcommand{\Etwo}{E^2}

\newcommand{\al}{\alpha}
\newcommand{\be}{\beta}
\newcommand{\de}{\delta}
\newcommand{\ga}{\gamma}
\newcommand{\eps}{\varepsilon}
\newcommand{\la}{\lambda}
\newcommand{\si}{\sigma}

\newcommand{\sm}{\setminus}
\newcommand{\wt}{\widetilde}

\newcommand{\Ga}{\Gamma}
\newcommand{\Om}{\Omega}

\newcommand{\supp}{\operatorname{supp}}
\newcommand{\Hs}{\mathcal{H}^s}

\newcommand{\Hscont}{{\mathcal{H}}^s_{\rm cont}}
\newcommand{\Honecont}{{\mathcal{H}}^1_{\rm cont}}

\newcommand{\Bxm}{B_{-}}
\newcommand{\Bxp}{B_{+}}

\newcommand{\Omc}{\Om^c}
\newcommand{\LtwomuH}{L^2(\mu, \Hdot)}

\begin{document}


\title[Spectral Dissection]
{Spectral dissection  of finite rank perturbations
of normal operators}

\dedicatory{To Dan Virgil Voiculescu on the occasion of his seventieth birthday}

\author{Mihai Putinar {\protect \and} Dmitry Yakubovich}
\address{Mihai Putinar, University of California at Santa Barbara, CA,
USA and Newcastle University, Newcastle upon Tyne, UK}
 \email{mputinar@math.ucsb.edu, mihai.putinar@ncl.ac.uk}

\address{Dmitry Yakubovich, Departamento de Matem\'aticas,
Universidad \newline Aut\'onoma de Madrid, and
ICMAT (CSIC - UAM - UC3M - UCM), Cantoblanco 28049, Spain}
\email{dmitry.yakubovich@uam.es}

 \thanks{
	The authors acknowledge partial support by Spanish Ministry of Science, Innovation and
	Universities (grant no. PGC2018-099124-B-I00),
	the ICMAT Severo Ochoa project SEV-2015-0554 of the Spanish Ministry of Economy and
	Competitiveness of Spain and the European Regional Development
	Fund. The first author expresses his gratitude to the Universidad Autonoma de Madrid for hospitality.
	We are both grateful to the Mittag-Leffler Institute for offering the opportunity to progress on this project.
}

\begin{abstract}
	Finite rank perturbations $T=N+K$ of a bounded normal operator $N$ acting on a separable Hilbert space
	are studied thanks to a natural functional model of $T$; in its turn the functional model solely relies on a perturbation matrix/ characteristic function previously defined
	by the second author.
	Function theoretic features of this perturbation matrix encode in a closed-form the spectral behavior of $T$.
	Under mild geometric conditions on
	the spectral measure of $N$ and some smoothness constraints on $K$ we show that the operator $T$ admits
	invariant subspaces, or even it is decomposable.
	
\end{abstract}






\maketitle

\section{INTRODUCTION}

Finite rank perturbations of Hilbert space operators were studied for at least a century, for instance for their far reaching connections
with function theory of a complex variable, boundary value problems of mathematical physics
or for applications to quantum theory. It is sufficient to mention the case of
dissipative operators with rank-one imaginary part or the celebrated phase-shift and
related perturbation determinant, see \cite{GK} for the golden era references, or \cite{Si} for more recent developments.
The booming topics of Aleksandrov-Clark measures \cite{PS} and the resurrection of
Aronszjan-Donoghue theory for matrix valued measures associated to finite rank perturbations of
self-adjoint operators \cite{LT,FL} are two other notable examples.
To name only one less known, additional relevant ramification: an apparently
non-related open problem of approximation theory, known as
Sendov conjecture, can be translated into spectral estimates of rank-two
perturbations of normal matrices, see \cite{KPPSS}.

Let $\cH$ be a complex, separable Hilbert space. The class of operators under study is
\beqn
\label{perturb-T}
T = N+ \sum_{k=1}^m \uu_j\otimes \vv_j = N + \sum_{k=1}^m \uu_j \langle \cdot, \vv_j \rangle,
\neqn
where $N$ is a normal operator on $\cH$ and the finite rank summand is subject to certain ``smoothness" conditions.
To fix ideas we first assume that $N$ has spectral multiplicity $1$. According to the Spectral Theorem, $N$ is unitarily equivalent to
the multiplication operator by the complex variable:
\[
Nf(z)=zf(z)\quad \text{on } \  \cH=L^2(\mu).
\]
In this functional model setting the announced ``smoothness'' properties of the functions
$\uu_j$, $\vv_j$ will be quite natural, for instance imposing their uniform boundedness. The case of general spectral multiplicity of $N$ is
covered by our  main results, with notation and conventions described in Section~\ref{the-model}.

The second author has already treated
smooth trace class perturbations
in the case of an absolutely continuous spectral measure (with respect to
area measure) \cite{Y1}.  A natural quotient model of $T$, defined in terms of
certain Sobolev-type spaces was introduced there. The present article builds on \cite{Y1} by proposing
a quotient functional model for a wide
class of spectral measures $\mu$. The concrete function theoretic features of the quotient model provide the announced spectral decomposition results.
In the present article we only treat finite rank perturbations due to inherent technical complications.

To give a preview of the nature of the functional model we propose in this work, we consider the simplest case
$$ T = M_z + u(z) \langle \cdot, v \rangle : L^2(\mu) \longrightarrow L^2(\mu),$$
with bounded measurable functions $u, v$ defined on the support $\sigma$ of the positive measure $\mu$.
Let $\Om$ be a domain with smooth boundary containing $\sigma$, so that the perturbation function
$$ \psi(z) = 1 +\int \frac{u(w)\overline{v(w)} d\mu(w)}{w-z},$$
does not vanish outside $\Om$.
Define the model space
${\rm Mod}(\Om)$ of Cauchy transforms plus analytic functions
$$  \int \frac{x(w)\overline{v(w)} d\mu(w)}{w-z} + b(z), \  z \in \Om,$$
where $x \in L^2(\mu)$ and $b(z)$ is analytic in the associated Hardy space $H^2(\Om)$.
This space is complete with respect to the Sobolev type norm
$$ \| x \|^2_{2,\mu} + \| b\|^2_{2, \partial \Om}.$$
The reader can immediately verify that the operator
$$ x \mapsto  \int \frac{x(w)\overline{v(w)} d\mu(w)}{w-z} $$
intertwines $T$ and multiplication by the complex variable on the quotient space
${\rm Mod}(\Om)/\psi {\rm Mod}(\Om).$
This precise similarity transform allows us to infer spectral decomposition properties of $T$ from function and measure theoretic results.

Our inquiry is affiliated to a series of recent developments. For instance, in a series of papers \cite{FJKP07}--\cite{FJKP11}, Foia\c{s}, Jung, Ko and Pearcy
considered rank one perturbations of normal operators with a
discrete spectral measure $\mu$. Their technique of
cutting the spectrum of the perturbation relied on carefully chosen integration contours applied to localized
resolvent, consequently producing appropriate Riesz projections. Later, their results were extended
by Fang and Xia \cite{FX}, and by Klaja \cite{Klaja-JOT} to finite rank and compact perturbations. Klaja also
has some results \cite{Klaja-JOT} for a non-discrete measure $\mu$, provided there is a Jordan curve $\Ga$ with certain
continuity and growth conditions
on the local resolvents $(N-\la)^{-1}u_j$ and $(N^*-\bar \la)^{-1}v_j$, where $\la\in \Ga$.

An earlier work on the same subject is \cite{Ionascu}.
And even before, studies of similar flavor go back to \cite{Ch, Sch}; see also the references in \cite{Y1}.

In \cite{Klaja-IEOT}, Klaja proves that if $N$ is diagonalizable and its spectrum is a perfect set,
then $N$ possesses a rank one perturbation without eigenvalues. If $N$ is compact and self-adjoint, this
is no longer true in general; the criterion for the existence of rank one perturbation of this type
was given in \cite{BY1}.

On a related, but totally different line of research, weakly convergent integrals of localized resolvents of operators belonging to a $II_1$ factor
appeared in the works of Haagerup and Schultz \cite{HS} and followers \cite{DSZ}.
These integrals provided Riesz type projections and ultimately unveiled a rich spectral
decomposition behavior.

Rank-one perturbations of normal operators were intensively studied during the last years
by Baranov and his collaborators (see \cite{ABB,Ba}), in particular, by
Baranov and the second author \cite{BY1}--\cite{BY3}. These works make use of a de Branges type functional model
and heavily rely on the theory of entire functions.

By far from being exhaustive, more references to old and new analyses
of finite rank perturbations of linear transforms are scattered through the present article.

Our aims are:

\begin{itemize}
	\item To find conditions for possibly ``dissecting'' the spectrum of $T$
	along a curve. Given a
	domain $\Om$ containing $\si(T)$, such that
	${\overline\Om}={\overline\Om}_1\cup {\overline \Om}_2$, the domains $\Om_1$ and $\Om_2$ are disjoint
	and have piecewise smooth boundaries,
	our Theorem~\ref{thm-direct-sum} asserts that, under additional
	conditions, there exists a direct sum decomposition $T=T_1\dotplus T_2$, satisfying the spectral localization conditions
	$\si(T_j)\subset \overline\Om_j$.

	\item To deduce from this fact sufficient conditions for the
	existence of
	invariant subspaces.
	
	\item To give sufficient conditions for the decomposability of $T$
	in the sense of Foia\c{s} \cite{Fo}.
\end{itemize}

While we do not seek in the present article the most general conditions imposed on the spectral measure $\mu$
and/or the finite rank perturbations, our approach gives rise to challenging questions in geometric function
theory. Some of these problems are
summarized in the last section.

The contents is the following. Section 2 contains some preliminaries of geometric measure theory and function theory of a complex variable.
There we introduce the concept of dissectible Borel measure to become the key technical tool for the rest of the article. Section 3 is devoted to the construction of the
Sobolev space type functional model for the perturbed operator. Our main spectral dissection theorem is stated in Section 3. In Section 4 we provide proofs for the geometric measure theory lemmas and in Section 5 we turn to the proofs
involving the functional model. Section 6 is focused on the existence criterion of a non-trivial invariant subspace for the perturbed operator. In Section 7 we recall some basic terminology and facts from local spectral theory; there we propose the notion of {\it dissectible operator}, slightly stronger than decomposable operator. In the same section
we provide sufficient criteria for the perturbed operator to be either dissectible, or only to possess Bishop's property $(\beta)$. In Section 8 we formulate a few open questions.

We dedicate this article to Dan Virgil Voiculescu whose exceptional insight into perturbation theory of linear operators has
reformed modern mathematical analysis. His contributions to the subject span more than four decades, to cite only \cite{Vo1,Vo2}.

\section{PRELIMINARIES}

This section collects a few function and measure theoretic results referred to in the sequel.
We also introduce some necessary notation and terminology.

Henceforth we rely on the standard dyadic system $\cD$ of squares
in $\C$ with sides parallel to coordinate axes.
The $k$th generation $\cD_k$
consists of squares of the form
$[m 2^{-k}, (m+1) 2^{-k}) \times [n 2^{-k}, (n+1) 2^{-k}) $; here
$k,m,n\in\Z$. Then
\[
\cD=\bigcup_{k\in\Z}\cD_k.
\]
A square $Q$ is dyadic if
$Q\in \cD_k$ for some $k$; in this case we set $|Q|=2^{-k}$ to be its
side length.

For any dyadic $Q\in\cD$ and any $c>0$, we denote
by $\overline{Q}$ the closure of $Q$,
by $Q^o$ its interior and by
$cQ$ the square homothetic to  $Q$ with the same center, such that
$|cQ|=c|Q|$.

We denote by $\Hs$ the $s$-dimensional Hausdorff measure and by
$\Hscont$ the $s$-dimensional Hausdorff content, calculated with
respect to the dyadic system~$\cD$.

\begin{definition}
	\label{def-s-nice1} Throughout this article $h$ stands for a
	nonnegative increasing function defined on $[0,+\infty)$
	satisfying $h(t)>0$ for $t>0$, $\lim_{\infty} h= +\infty,$ and
	\beqn \label{int-condn-h} \int_{0}^{1} t^{-2} h(t)\, dt <\infty.
	\neqn Such a function $h$ is called \textit{a scale function}.
\end{definition}

The main example is
\[
h(t)=t^\aaaa,
\]
where $1<\aaaa\le 2$.
To capture finer effects (related to measures
``close'' to the arc length measure in the sense of dimension),
one can also consider functions of the form
$h(t)= t\cdot \log^{-a} (2+1/t)$, $a>1$.

\begin{definition}
	Let $\nu$ be a finite positive measure on $\C$ with compact support.
	
	1) We say that a dyadic square $Q$ is
	\emph{light (w.r. to $\nu$)} if
	\[
	\nu(Q)\le h(|Q|),
	\]
	and \emph{heavy} in the opposite case.
	
	2) Given a measure $\nu$, consider the increasing
	family of open sets
	\begin{equation}
	\label{eq3n}
	\cG_s(\nu)=\bigcup_{Q \; \text{dyadic and heavy}, \;
		|Q|\le s} \; 10 Q^o, \quad \text{where } s>0.
	\end{equation}
	We say that $\nu$ is \emph{dissectible} if
	\[
	\lim_{s\to 0^+} \Honecont(\cG_s(\nu)) = 0.
	\]
	In this case, we will also use the terms $h$-\emph{dissectible},
	or $\aaaa$-\emph{dissectible}, for the case $h(t)=t^\aaaa$.
\end{definition}

Notice that condition
\[
\sum_{Q \;\text{dyadic and heavy}} \, |Q|<\infty
\]
is sufficient for $\nu$ to be $h$-dissectible.

We remark that for any finite measure $\nu$,
every sufficiently large square in $\cD$ is light.
This implies the following observation.
Suppose $\nu$ is a $h$-dissectible measure as above. Choose any function
$h_1$ meeting the same conditions as $h$ and such
that $h_1(t)/h(t)\to \infty$ as
$t\to 0^+$.
Then for any nonnegative function $f\in L^\infty(\nu)$,
$f\nu$ is $h_1$-dissectible.

In what follows, we
set
\beqn
\label{nu}
\nu=\big(\sum_j |\uu_j|^2+|\vv_j|^2\big)\mu.
\neqn
Most of our results require the following condition:

\bigskip

\begin{enumerate}
	
	\item[(D)]
	{\bf The measure $\nu$ is dissectible.}
\end{enumerate}

\bigskip

Given a complex measure $\tau$ of finite total variation, we put
\begin{equation}
\label{eq1n}
C(\tau)(z) :=\int \frac{ d\,\tau(t)}{t-z}
\end{equation}
to be its Cauchy transform.
It is known that $C(\tau)$ is defined at least as an element of $L^1_\text{loc}(\C)$.

As shown by Verdera \cite{V} and by Mattila and Melnikov \cite{MM},
for any measure $\nu$ and any Ahlfors-David regular curve $\ga$,
the Cauchy integral $C(\nu)$ exists a.e. on $\ga$ and satisfies
a weak $L^1$ estimate on $\ga$. 
Our constructions below require a stronger regularity
of the Cauchy integrals $C(f\nu)$.
In particular, we will use the following lemma,
which reveals the importance of dissectible measures in our study.

\begin{lemma}
	\label{L-infty-lemma}
	For any positive measure $\nu$, any $s\in (0,1]$ and any
	$f\in L^\infty(\nu)$, one has
	\beqn
	\label{estim-eta}
	|C(f\nu)(z)|\le K s^{-1} \|f\|_{\infty}, \quad z\in\C\setminus \cG_s(\nu),
	\neqn
	where $K$ is a constant only
	depending on $\nu$.
	Moreover, $C(f\nu)$ is a continuous function on $\C\setminus \cG_s(\nu)$.
\end{lemma}

In particular, if $\nu$ is a dissectible measure, then $C(f\nu)$ is continuous
on sets with ``small'' complement.

If $\Om$ is a Jordan domain with piecewise $C^1$ smooth boundary, we
will use the abbreviation
\[
C_{\pt\Om}(w):=
C(w\cdot dz\big|_{\pt\Om}),
\]
(here $dz$ corresponds to the positive orientation of $\pt\Om$).
We adopt the notation $L^p(\pt\Om)=L^p(\pt\Om, |dz|)$.

If $\Om$ is a Jordan domain in the complex plane, by an
\textit{exhaustion} of $\Om$ we mean an increasing sequence of Jordan
domains $\Om_n$ with rectifiable boundaries such that $\cup_n
\Om_n=\Om$. For $1\le p < \infty$,  $E^p(\Om)$ denotes the
Smirnov space of analytic functions in $\Om$. We recall that a
function $f$, holomorphic in $\Om$, belongs
to $E^p(\Om)$ if there
is an exhaustion $\{\Om_n\}$ of $\Om$ such that
\beqn
\label{int-Om-n}
\sup_n \int_{\pt\Om_n}|f(z)|^p\, |dz|
<\infty.
\neqn
These are Banach spaces for any $p$ as above; for $p=2$, they are Hilbert spaces.
The space $E^\infty(\Om)=H^\infty(\Om)$ is just the space of bounded analytic functions on $\Om$.

We recall some basic properties of Smirnov spaces (see \cite{Du, Z}).
Denote by $K(\pt\Om)$
\textit{the David constant}
of the curve $\pt\Om$, that
is, the least constant such that $\cH^1(\pt\Om\cap B(z,r))\le Kr$
for all disks $B(z,r)$ in the plane. If $\pt\Om$ is sufficiently
good (say, locally Lipschitz), then one can use
the same exhaustion $\{\Om_n\}$ in \eqref{int-Om-n} for all
functions $f$. Namely, take any exhaustion $\{\Om_n\}$ of $\Om$
with uniformly bounded David constants $K(\pt\Om_n)$ (such a
sequence of domains always exists). Then a function $f$,
holomorphic in $\Om$, belongs to $E^p(\Om)$ if and only if
\eqref{int-Om-n} holds for this sequence of domains.

Let $\Om^c=\wh\C\sm{\overline\Om}$ be the complementary domain, where $\wh\C$ is the Riemann sphere.
For a function $f$ belonging to $E^2(\Om)$ or to $E^2(\Om^c)$,
its boundary values on $\pt\Om$
are well-defined as an element of
$L^2(\pt\Om)=L^2(\pt\Om, |dz|)$.
In this sense, the spaces $E^2(\Om)$ and $E^2(\Om^c)$ can be identified with closed subspaces
of $L^2(\pt\Om)$. The direct sum decomposition
\beqn
\label{smirn-class-decomp}
L^2(\pt\Om)= E^2(\Om)\dotplus E^2_0(\Om^c)
\neqn
holds; here $E^2_0(\Om^c)=\{f\in E^2(\Om^c):f(\infty)=0\}$.
In general the above direct sum is not orthogonal.
However, the parallel projections onto the
two direct summands in~\eqref{smirn-class-decomp} are given by
the linear bounded transformations
$f\mapsto C_{\pt\Om}f\big|_{\Om}$ and
$f\mapsto -C_{\pt\Om}f\big|_{\Om^c}$.

\begin{definition}
	Given a measure $\mu$ on $\C$ and a
	subset $E\subset \C$, define the corresponding \emph{Carleson constant}
	of $E$ with respect to $\mu$ by
	\[
	\Carl(E,\mu):=\sup_{z\in E, r>0} \frac {\mu(B(z,r))} r,
	\]
	where $B(z,r)$ is the open disc in $\C$ of radius $r$, centered at $z$.
\end{definition}

Suppose $\Om$ is a domain satisfying $\mu(\pt\Om)=0$. 
Consider the \emph{Carleson's embedding operator}: 
\beqn
\label{Carleson-J}
J_\Om:E^2(\Om)\to L^2(\mu), \quad J_\Om f:=f.
\neqn
It is known \cite{Du, Z} that both operators $J_\Om$ and $J_{\Om^c}$ are bounded if and only if
Carleson's constant of the boundary $\pt\Om$ with respect to $\mu$ is bounded.

\begin{definition}
	\label{def-adm-domain}
	A Jordan curve $\ga$ is called {\it admissible} (with respect to $\mu$)
	if it satisfies the following conditions:
	
	\begin{enumerate}
		\item $\ga$ is a piecewise $C^1$ smooth curve;
		
		\item
		Let $\Om$ and $\Om^c$ be the two connected components of $\wh\C\sm\ga$.
		There exists an exhaustion $\{\Om_n\}$ of the domain $\Om$
		such that all curves $\ga$ and $\pt\Om_n$
		are contained in $\C\sm \cG_s(\mu)$ for some (fixed) $s>0$ and all
		David constants $K(\pt\Om_n)$ as well as
		Carleson constants $\Carl(\pt\Om_n, \mu)$ are uniformly bounded.
		The same property holds for $\Om^c$.
	\end{enumerate}
\end{definition}

If $\ga$ is admissible, we simply call $\Om$ and $\Om^c$
\emph{admissible domains}.

\begin{rem}
	It is easy to see that $\mu(\ga)=0$ whenever
	the Jordan curve $\ga$ does not intersect $\cG_s(\mu)$ for some $s>0$.
\end{rem}


\begin{lemma}
	\label{lem-broken-line}
	Let $\mu$ be any compactly suppported finite measure on $\C$. Then there 
	are subsets \;$\A_x, \A_y$\; of the real line
	such that $\cH^1(\R\sm \A_x)=
	\cH^1(\R\sm \A_y) = 0$
	possessing the following properties:
	
	1) Any Jordan broken line $\ga$ is admissible whenever it consists
	of finitely many intervals parallel to the axes and
	its vertices belong to $\A_x +i\A_y$.
	
	2) For any $x_0\in \A_x$, there is an increasing sequence
	$\{x_n^-\}$
	and a decreasing sequence
	$\{x_n^+\}$
	both tending to $x_0$ and such that
	for some $s>0$, all the points $x_0, x_n^-$ and $x_n^+$ belong
	to $\A_x\sm \Re \cG_s(\mu)$. A similar approximation property holds for the set $\A_y$
	(with $\A_y\sm \Im \cG_s(\mu)$ replacing $\A_x\sm \Re \cG_s(\mu)$).
\end{lemma}

\begin{examples}
	\label{examples-measures}
	1) Let $\mu$ be the area measure restricted to
	some bounded Borel subset of $\C$. Let $d\nu=u\,d\mu$, where
	$u\in L^p(\mu)$. Let $1/p+1/q=1$. If $p>2$, then by H\"older's inequality,
	any dyadic square is light:
	\[
	\nu(Q)\le \mu(Q)^{1/q}\bigg(\int_Q |u|^p\,d\mu \bigg)^{1/p}\le
	\|u\|_p |Q|^{2/q},
	\]
	and $\aaaa=2/q>1$.
	So each measure $\nu$ of this form is $\aaaa$-dissectible
	(as it was mentioned, the constant $\|u\|_p$ does not matter).

	\medskip
	
	2)
	Similarly, assume that
	$\mu$ satisfies an estimate
	$\mu(Q)\le C|Q|^\si$ for any square $Q$, where
	$\si\in [1,2)$ and $C$ are constants.
	Choose $q\in(1,\sigma)$ and
	the corresponding $p\in(1,\infty)$.
	Then that a measure $d\nu=u\,d\mu$, where the density
	$u$ is in $L^p(\mu)$, is $a$-dissectible for $a=\sigma/q>1$.
	This is obtained in the same way as in the previous example.
	
	%
	%
	
	\medskip
	
	3) On the opposite side, if the closed support $\supp \mu$
	has Minkowski dimension
	less than~$1$, then it is easy to see that $\mu$ is $\aaaa$-dissectible
	for any $\aaaa>1$ (a square can be heavy only if it touches $\supp \mu$).
	\medskip

	4) Of course, the closed support is a very rough
	characteristic of ``the dimension'' of  $\mu$.
	Suppose, however, that measures $\mu_k$, $k\ge 0$, are mutually singular
	and $h_k$-dissectible, for some scale functions $h_k$. Suppose $\sum\mu_k(\C)$ is finite. Put
	$\mu=\sum\mu_k$. This measure can fail to be dissectible, but
	there is an equivalent measure to $\mu$ that is
	$h$-dissectible for some $h$. Namely, the following Proposition holds.

	\begin{prop}
		\label{sum-dissect-measures}
		Suppose that measures $\mu_k$ are $h_k$-dissectible for some
		scale functions $h_k$,
		and suppose that the union of the closed supports
		of these measures is a bounded subset of $\C$.
		Then there are (small) positive constants $c_k$ such that
		the measure $\sum_k c_k \mu_k$ is finite and dissectible with
		respect to some scale function $h$.
	\end{prop}
	
	5) In particular, for any atomic measure $\mu=\sum \de_{t_k}$ there is
	an equivalent measure $\nu=\sum c_k\de_{t_k}$, which is dissectible.
	This can be related to the techniques proposed by
	Foia\c{s} and his collaborators \cite{FJKP07}--\cite{FJKP11}.
\end{examples}

The proofs of Lemmas~\ref{L-infty-lemma}, ~\ref{lem-broken-line}
and Proposition~\ref{sum-dissect-measures} will be given in Section~\ref{sect-geom}.

\begin{rem}
	A measure $\nu$ is not $h$-dissectible (for any scale function $h$)
	whenever for some Ahlfors-David curve $\ga$,
	$\nu|\ga$ has a nontrivial part, absolutely continuous with respect
	to arc length.
	This observation can be deduced from
	Lemma~\ref{L-infty-lemma}.
\end{rem}

\smallskip
One can infer from the
above examples that measures of ``dimension'' $1$ are the
most difficult case, at least for this approach.

\section{THE FUNCTIONAL MODEL}
\label{the-model}

We embark now on the higher spectral multiplicity framework. Specifically,
the Spectral Theorem implies that by passing to a unitarily equivalent operator, one
can assume that the normal operator $N$ is represented by
a von Neumann direct integral:
\beqn
\label{vN-repres-N}
Nf(z)=zf(z), \quad f\in \cH=L^2(\mu, \h):=\int^\oplus \h(z)\, d\mu(z),
\neqn
where $\{\h(z)\}$ is a measurable family of Hilbert spaces, referred to as fibers.

If the fiber space is constant, that is, $H(z)\equiv L$ $\mu$-a.e., then $\cH=L^2(\mu)\otimes L$.
This is why we adopt the shorter
notation $L^2(\mu, \h)$ for the direct integral in~\eqref{vN-repres-N}.

Put $\hdot(z)=\spann\{u_j(z), v_j(z): 1\le j\le m\}$; then
$\hdot(z)\subset H(z)$ and
$\dim \hdot(z)\le 2m$ $\mu$-a.e.
We make the following second standing assumption:

\bigskip

\begin{enumerate}
	
	\item[(K)]
	{\bf For $\mu$-almost every $z\in \C$,
		the vectors $v_j(z)$, $j=1,\dots, m$ generate the space~$\hdot(z)$.}
\end{enumerate}

\bigskip

Any finite rank perturbation can be rewritten in a form that
satisfies (K). Indeed, it suffices to add to the
expression~\eqref{perturb-T} finitely many new formal terms $\langle
\cdot, \vv_j\rangle\,\uu_j$ with $\uu_j=0$.

From now on, we build the quotient functional model only for the restriction
\[
T_0:=T|\cH_0
\]
of $T$ to its reducing subspace
\beqn
\label{Hdot}
\cH_0:=L^2(\mu, \hdot).
\neqn
Notice that the restriction of $T$ to $\cH\ominus \cH_0$ is normal and is already represented
in its von Neumann functional model.

Let $x, y\in \cH_0$ be a pair of vectors
(so that $x(z), y(z)$ are measurable cross-sections, $x(z), y(z)\in \hdot(z)$, and
$\|x(\cdot)\|, \|x(\cdot)\|\in L^2(\mu)$).
Their $\odot$ product is defined by:
\[
\big(x\odot y\big)(z):= [ x(z), y(z) ]_{\hdot(z)}
\]
(so that $x\odot y\in L^1(\mu)$).
To simplify notation,
we assume henceforth that
the products $[ x(z), y(z) ]_{\hdot(z)}$ are bilinear.
To be more precise, fix a measurable family of orthonormal bases $e_j(z)$, $1\le j\le n(z)$
in spaces $\hdot(z)$ and define these products in terms of corresponding coordinates:
\[
[ x(z), y(z) ]_{\hdot(z)}= \sum_{j=1}^{n(z)} x_j(z) y_j(z).
\]
The usual sesquilinear product in a fiber $\hdot(z)$ (used in~\eqref{perturb-T})
is given by
\[
\langle x(z), y(z) \rangle_{\hdot(z)}= [ x(z), \bar y(z) ]_{\hdot(z)}.
\]
Here $\bar y(z)=\sum_j \bar y_j(z) e_j(z)$ whenever
$y(z)=\sum_j y_j(z)e_j(z) \in \hdot(z)$.

In this situation, in the definition~\eqref{nu} of
the measure $\nu$, we set $|u_j(z)|:=\|u_j(z)\|_{\h(z)}$, and the
same for $|v_j(z)|$.

If $H_0(z)\equiv L$ $\mu$-a.e., where $\dim L=1$, then $\cH_0$ can be identified with $L^2(\mu)$,
and $x\odot y$ is just the pointwise product of $L^2$ functions $x$ and $y$.

We will use the formal columns $\uu=(\uu_1, \dots, \uu_m)^t$ and
$\vv=(\vv_1, \dots, \vv_m)^t$.
Then the perturbation in~\eqref{perturb-T}
can be expressed as follows:
\beqn
\label{perturb}
\sum_{j=1}^{m} \langle \cdot, \vv_j\rangle\,\uu_j =
\uu^t \langle \cdot, \vv\rangle.
\neqn

In what follows, instead of condition (D) we impose for convenience the following two conditions, that are formally stronger:
\bigskip

\begin{enumerate}
	\item[(D$'$)]
	{\bf The measure $\mu$ is dissectible};
	\bigskip
	
	\item[(B)]
	{\bf The functions $u_j$ and $v_j$ are bounded}.
\end{enumerate}
\bigskip

If (K) holds, there is no loss of generality in assuming both (D$'$) and (B): one can
achieve it by substituting $\mu$ with an equivalent measure. Indeed,
the function $\rho=\sum_j |u_j|^2+|v_j|^2\in L^1(\mu)$
is positive and, by (K), nonzero $\mu$-a.e.
Then the data
\[
\wt\mu := \rho \mu, \;
\wt u_j := \rho^{-1/2} u_j, \;
\wt v_j :=  \rho^{-1/2} v_j
\]
satisfy (D$'$) and (B) and define new
operators $N$ and $T$, which are unitarily equivalent to the original ones.

The construction of our functional model does not require condition (D$'$). However, in order to achieve our spectral decomposition goals the
dissectability of the spectral measure is necessary.

\begin{lemma}
	\label{lem-Cauchy-mu}
	Let $\Om$ be a Jordan domain with admissible boundary with respect to $\mu$.
	Formula
	\[
	f\mapsto C(f\mu)
	\]
	defines bounded operators defined on $L^2(\mu|\Om)$
	and with values on $E^2(\Omc)$ and similarly  from $L^2(\mu|\Omc)$
	to $E^2(\Om)$.
\end{lemma}

\begin{proof}
	Let $f\in L^2(\mu)$ and assume that the closed support of $f$
	does not touch $\pt\Om$.
	For any function $g\in E^2(\Om)$, analytic on a neighbourhood of
	the closure of $\Om$, one has
	\[
	\frac 1{2\pi i} \int_{\pt\Om} C(f\mu)\cdot g\, dz =
	\int_\Om fg\, d\mu.
	\]
	Since $\Om$ is a Smirnov domain~\cite{Du}, functions $g$ as above form a dense subspace of $E^2(\Om)$.
	
	Notice also that $E^2(\Om)$ is dual to
	$\Etwo_0(\Omc)$
	with respect to Cauchy duality (defined by
	the left hand side of the above identity). Accordingly,
	the boundedness of Carleson's  embedding operator ~\eqref{Carleson-J}
	implies the boundedness of the linear transform $f\mapsto C(f\mu)$, defined on
	$L^2(\chi_{\Om}\cdot\mu)$ and with values in $\Etwo_0(\Omc)$. The boundedness of the second operator is
	checked similarly.
\end{proof}

\begin{definition}
	\label{def-model-space}
	To a domain $\Om$ in $\C$ with admissible boundary we associate
	the following \emph{model space}
	\beqn
	\label{Mod-Om}
	\Mod(\Om) = \big\{g=C (x\odot \bar v\cdot\mu) + b:
	\quad  x\in L^2(\mu|\Om, H_0),
	\; b \in \Etwo(\Om)\otimes \C^m \big\},
	\neqn
	where $x\odot \bar \vv = (x\odot \bar \vv_1, \dots, x\odot \bar \vv_m)^t$
	is a column function. According to (D$'$), $x\odot \bar \vv\in L^2(\mu)\otimes \C^m$.
\end{definition}

This space is contained in $L^1(\Om, dA)\otimes \C^m$, where $dA$ is the area
measure. It follows from (K) that the function $x$ is determined
by $g$ uniquely from $x\odot \bar v\cdot\mu = -\pi \dbar g$ on $\Om$ (in the sense
of distributions). Hence $b$ is also determined by $g$.

\begin{prop}
	\label{prop-traces}
	Let $\Om$ be an admissible domain
	and $\ga$ an admissible Jordan curve,
	which is either contained in $\Om$ or is a subarc of $\pt\Om$.
	
	Then any function $g$ in $\Mod(\Om)$ has a well-defined ``trace''
	$g|\ga$ on $\ga$, which is an element of
	$L^2(\ga)\otimes\C^m$. The map $g\mapsto g|\ga\in L^2(\ga)\otimes\C^m$ is bounded.
\end{prop}

The exact definition of these traces will be given in Section~\ref{sect-proofs-model}.

Formula
\[
\|g\|^2:=\|x\|^2_{L^2(\mu|\Om, H_0)}+\|b\|^2_{\Etwo}
\]
defines a Hilbert space structure on $\Mod(\Om)$ (we leave the
details to the reader). Proposition~\ref{prop-traces} implies that the maps
\beqn
\label{bded-maps-mod-space}
g\mapsto g|\partial\Om \in L^2(\pt\Om)\otimes \C^m, \quad g\mapsto x=x(g)\in L^2(\mu, H_0)
\neqn
are bounded on $\Mod(\Om)$. Hence formula
\beqn
\label{equiv-norm-mod-space}
\|g\|_1^2:=\|x(g)\|^2_{L^2(\mu|\Om, H_0)}+\|g\|^2_{L^2(\pt\Om)}
\neqn
defines an equivalent Hilbert space norm on $\Mod(\Om)$.

The $m\times m$ matrix function
$\Psi$, defined by
\[
\Psi=I_m+C(\bar v \odot u^t\cdot \mu)
\]
is called \textit{the perturbation matrix} and will play
a major role in the sequel. It belongs to $L^2_{loc}(\C, dA)$.
By Lemma~\ref{L-infty-lemma}, this function is defined and continuous  on $\wh \C\sm \cG_s(\mu)$,
for any $s>0$.

Notice that $\Psi$ might be not continuous on the complement of the union of the sets $\cG_s(\mu)$.

Put
\[
\psi=\det\Psi;
\]
this scalar valued function is called \textit{the perturbation determinant}.
We observe that $\Psi$ and $\psi$ are holomorphic on $\wh\C\sm \supp\mu$,
$\Psi(\infty)=I_m$ and $\psi(\infty)=1$.

For a domain $\Om$ as in the above Definition, set
\beqn
\label{def-W-0-Om}
W_{0,\Om}x(z)= C\big(x\odot \bar v\cdot \mu\big)(z), \quad z\in\Om,
\neqn
so that
\[
W_{0,\Om}: L^2(\mu|\Om, H_0)\to \Mod(\Om).
\]

A direct calculation yields the intertwining property
\beqn
\label{intertw-W0-new}
W_{0,\Om} T x(z) = z  W_{0,\Om} x(z) + \Psi(z)\,\langle x, v\rangle.
\neqn

We denote by
$\Der\supp\mu$ the derivative set of the support of $\mu$, that is,
the set of all accumulation points of $\supp\mu$.

\begin{prop}
	\label{prop-spectrum-T}
	
	Set $T_0=T|\cH_0$.
	\begin{enumerate}
		
		\item
		The essential spectrum of $T_0$ coincides with the set $\Der\supp \mu$.
		
		\item
		Suppose  $\la\in \C\sm\supp \mu$. Then $\la\in\si(T_0)$ iff $\psi(\la)=0$.
		The same criterion guarantees $\la$ to belong to the point spectrum of $T_0$.
		
	\end{enumerate}
\end{prop}

\begin{prop}
	\label{prop-closed-ss}
	For any admissible domain $\Om$ such that $\psi\ne 0$ on $\pt\Om$,
	the linear manifold
	\[
	\Psi \cdot\big(\Etwo(\Om)\otimes \C^m\big)
	\]
	is a closed subspace of $\Mod(\Om)$.
	Moreover, for any $a\in \Etwo(\Om)\otimes \C^m$, the trace of $\Psi a$ on $\pt\Om$, in the sense of
	Proposition~\ref{prop-traces}, equals $\Psi\cdot a|\pt\Om$ (notice that
	$\Psi$ is continuous on $\pt\Om$).
\end{prop}

\begin{lemma}
	\label{lemma-multiplication-oper}
	Suppose $\Om$ is a domain with admissible boundary. Then the operator
	\beqn
	\label{f-mapsto-zf}
	f(z)\mapsto zf(z)
	\neqn
	is bounded on $\Mod(\Om)$. Its spectrum is contained in
	$\overline\Om$.
\end{lemma}

For any domain $\Om$ as in the above Proposition~\ref{prop-closed-ss},
we consider the quotient space
\[
\cQ(\Om):= \Mod(\Om)\big/\Psi \cdot(\Etwo(\Om)\otimes \C^m)
\]
and call it
\textit{the quotient model space corresponding to $\Om$}.
The quotient multiplication operator ~$M_\Om$,
induced by the above mapping~\eqref{f-mapsto-zf},
is correctly defined  and bounded on $\cQ(\Om)$.
We state the following analogue of Theorem 1 in \cite{Y1}.

\begin{theorem}
	\label{thm-model}
	Assume conditions (B) and (K) hold true. Let $\cB$ be a domain with
	admissible boundary, such that
	$\si(T)\subset\cB$. Then the transform
	\[
	W_\cB x:=
	C( x \odot\bar v\cdot \mu), \qquad W_\cB: L^2(\mu, \h_0)\to \cQ(\cB),
	\]
	is an isomorphism. It intertwines the operator $T_0$ with the quotient multiplication operator ~$M_\cB$.
\end{theorem}

This theorem shows that the perturbation
matrix $\Psi$ is a close analogue of the characteristic
function of a contractive or dissipative linear operator. This function
appeared first in 1946 in the paper by Livsi\v{c} \cite{Livsic}, dedicated to
quasi-hermitian operators with defect indices $(1,1)$.

The following two statements are analogous to \cite{Y1}, Lemmas~2 and 6.

\begin{lemma}[Glueing lemma]
	\label{lemma-Glueing}
	Suppose $\Om_1$, $\Om_2$ and $\Om$ are admissible domains,
	$\overline{\Om}= \overline{\Om}_1\cup  \overline{\Om}_2$, and
	$\overline{\Om}_1\cap  \overline{\Om}_2$ is an arc. If
	$w_j\in \Mod(\Om_j)$ and $w_1=w_2$ on $\overline{\Om}_1\cap  \overline{\Om}_2$
	(in the sense of Proposition~\ref{prop-traces}), then
	$w_j=w|\overline{\Om}_j$ for a function
	$w\in \Mod(\Om)$.
\end{lemma}

\begin{theorem}
	\label{thm-direct-sum} Assume conditions (B) and (K) are satisfied.
	Suppose $\cB$ is an admissible domain which contains $\si(T)$, 
	and let
	$\overline{\cB}=\cup_{j=1}^N \overline{\Om}_j$ be its finite
	partition, where $\Om_j$ are open and disjoint admissible
	domains. Assume that the union of boundaries $\pt\Om_j$ does not
	intersect $\cG_s(\mu)$ for some $s>0$ and that the
	perturbation determinant $\psi$ does not vanish on $\cup_{j=1}^N \,\pt\Om_j$.
	
	Then $T$ splits into a direct sum
	\[
	T=\dotplus_{j=1}^N T_j,
	\]
	where for each $j$, the spectrum of $T_j$ is contained in the closure
	of~$\Om_j$.
\end{theorem}

Formally, the above theorem does not require the dissectability
condition (D$'$). Its main application, however, is for measures
$\mu$ meeting this condition.
In this case,
the lengths of the sets $\Re \cG_s(\mu)$ and
$\Im \cG_s(\mu)$ tend to zero as $s\to 0$.
Therefore, by Lemmas ~\ref{L-infty-lemma} and ~\ref{lem-broken-line},
Theorem~\ref{thm-direct-sum} permits one to dissect the spectrum along straight lines
that cover densely the whole plane. This property is reflected in the
abstract notion of a dissectible Banach space operator,
which we introduce in Section~\ref{sec-Bishop-decomp-diss}.
Theorem~\ref{corr-thm-direct-sums}, which is a
consequence of the above Theorem~\ref{thm-direct-sum},
gives a sufficient condition for a finite rank perturbation of a normal operator to be dissectible.

\section{PROOFS OF LEMMAS FROM GEOMETRIC FUNCTION THEORY}
\label{sect-geom}

In the present section we prove Lemmas~\ref{L-infty-lemma}, ~\ref{lem-broken-line}
and Proposition~\ref{sum-dissect-measures}.

\begin{lemma}
	\label{lem1}
	For any real number $x$, there is a strictly
	increasing sequence $\{x_k^-: k\ge 0\}$ of numbers of the form
	$x_k^-=m_k/2^k$, $m_k\in \Z$, such that
	
	1) $\lim_k x_k^-=x$;
	
	2) $2^{-k}\le x - x_k^- \le 2^{-k+1}$.
\end{lemma}

\begin{proof}
	Set $s_k$ to be the
	truncation of the binary representation of $x$ with
	$k$ digits after the point. Then
	$\{s_k\}$ increases, tends to $x$, and
	$0\le x - s_k \le 2^{-k}$.
	Now put $x_k^-=s_k-2^{-k}$.
\end{proof}

\begin{proof}[Proof of Lemma \ref{L-infty-lemma}]
	We can assume that $s=2^{-k_0}$ for some
	$k_0\ge 1$.
	Take any $f\in L^\infty(\nu)$ with
	$\|f\|_{L^\infty(\nu)}\le 1$ and set $\eta=\eta_f=C(f\nu)$.
	Put $z=x+iy$, and let $\{x_k^-\}_{k\ge 1}$ be the sequence approximating $x$ from
	below, constructed in Lemma~\ref{lem1}  and
	$\{x_k^+\}$ the sequence
	approximating $x$ from above, with analogous properties.
	Similarly, take two sequences $\{y_k^-\}, \{y_k^+\}$, approximating $y$ from below
	and from above with the same properties.
	We have $x_k^- < x < x_k^+$,
	$2^{-k}\le |x-x_k^\pm|\le 2^{-k+1}$, and
	the same for $y$ and $\{y_k^\pm\}$. Consider
	the nested sequence of rectangles
	\[
	R_k(z)=[x_k^-,x_k^+)\times [y_k^-,y_k^+), \quad k\ge k_0-1,
	\]
	whose intersection is the point $z$.
	For any $k\ge k_0$,
	$R_{k-1}(z)\sm R_{k}(z)$ is a union of no more than
	16 dyadic squares from the generation $\cD_{k}$.
	The (dyadic) squares $Q$ satisfying
	\[
	Q\in \cD_k \quad \text{and}
	\quad Q\subset
	R_{k-1}(z)\sm R_{k}(z)
	\quad \text{for some } k\ge k_0
	\]
	are disjoint, and their union
	over all $k\ge k_0$
	equals to $R_{k_0-1}(z)\sm \{z\}$.
	We arrange them in a single sequence $\{Q_m\}_{m\ge 1}$.

	Notice that any square in the sequence
	$\{Q_m\}$ is at moderate distance from $z$ in the sense that
	\begin{equation}\label{moder-dist}
	|Q_m|\le \dist(Q_m, z)\le 2^{3/2} |Q_m|.
	\end{equation}
	
	Since
	$z\notin \cG_s(\nu)$ and $|Q_m|\le 2^{-k_0}$ for all
	$m\ge 1$, it follows that none of the squares
	$Q_m$ is heavy (see \eqref{eq3n}). We also observe that
	$\nu(\{z\})=0$ (otherwise any sufficiently small
	dyadic square containing $z$
	would be heavy, which would imply
	that $z\in \cG_s(\nu)$).
	Whence
	\begin{equation}
	\label{eq5}
	\begin{aligned}
	|\eta_f(z)|  &
	\le
	\int_{\C\sm R_{k_0-1}(z)} \frac{|d\nu(t)|}{|t-z|}
	+
	\sum_{m} \int_{Q_m} \frac{|d\nu(t)|}{|t-z|}       \\
	& \le 2^{-k_0+1}\, \nu(\C) +
	\sum_{m} \frac {\nu(Q_m)}{|Q_m|}   \\
	& \le 2^{-k_0+1}\, \nu(\C) +
	\sum_{m} \,  \frac {h(|Q_m|)}{|Q_m|}  \\
	& \le 2^{-k_0+1}\, \nu(\C) +
	16 \GH (2^{-k_0}),
	\end{aligned}
	\end{equation}
	where $\GH(\tau)=\int_0^\tau \frac {h(t)}{t^2}\, dt<\infty$.
	This proves the estimate \eqref{estim-eta}.

	With a small extra effort, we also derive the continuity of
	$\eta_f$ on $\C\sm \cG_s(\nu)$. Fix some $z\in \C\sm \cG_s(\nu)$, and
	let us prove the continuity of $\eta_f | \C\sm \cG_s(\nu)$ at $z$.
	It suffices to consider the case when $z\in (\supp \nu)\sm
	\cG_s(\nu)$.
	
	First we isolate the following observation.
	Take some $w\in B(z,r)\sm \cG_s(\nu)$, where the radius
	$r$ is small.
	Choose
	$k_1$ so that $2^{-k_1-2}\le r<2^{-k_1-1}$, and
	repeat the above construction of squares $Q_m$ and
	the estimate \eqref{eq5}, taking $k_1$ in place of $k_0$.
	We infer that,
	for any $f$ as above and any $\eps>0$, there exists
	$r>0$ such that
	\[
	\Big|\int_{B(z,r)}\frac{f(t)d\nu(t)}{t-w}\Big| <
	\Big|\int_{R_{k_1-1}(z)}\frac{f(t)d\nu(t)}{t-w}\Big| <
	16 \GH(2^{-k_1}) < \frac \eps 3
	\]
	whenever $w, z\in B(z,r)\sm \cG_s(\nu)$.
	
	We check the continuity at $z$ directly by definition.
	Fix $\eps>0$.
	Let $w\in \C\sm \cG_s(\nu)$, $|z-w|<\de$, where $\de$ is to be determined.
	Choose $r$ as above, and assume $\de<r$.
	Then
	\beqn
	\label{est}
	\begin{aligned}
		|\eta_f(z)-\eta_f(w)|
		& \le
		\Big|\int_{B(z,r)}\frac{f(t)d\nu(t)}{t-z}\Big|
		+
		\Big|\int_{B(z,r)}\frac{f(t)d\nu(t)}{t-w}\Big| \\
		& \qquad  +
		\bigg|\int_{\C\sm B(z,r)}\, f(t)d\nu(t)[(t-z)^{-1}-(t-w)^{-1}]\,\bigg|\, .
	\end{aligned}
	\neqn
	Since the last integral is continuous
	as a function of $w$ on the open disk
	$B(z,r)$, it follows that $|\eta_f(z)-\eta_f(w)|<\eps$ whenever
	$w\in \C\sm \cG_s(\nu)$, $|z-w|<\de$ and $\de$ is sufficiently small.
	That is, $\eta_f$ is continuous on $\C\sm \cG_s(\nu)$.
\end{proof}

Next we pass to the proof of Lemma~\ref{lem-broken-line}.
We recall that a point $x_0$ is called a (Lebesgue) density
point of a measurable set $A\subset \R$ if
$x_0\in A$ and
$\cH^1\big([x_0-\eps, x_0+\eps]\cap A\big)/(2\eps)\to 1$ as $\eps\to 0^+$.
In virtue of Lebesgue Density Theorem,
given a subset $A$ of $\R$ of positive measure,
almost every point of $A$ is a density point.

\begin{proof}[Proof of Lemma~\ref{lem-broken-line}]
	It is easy to see that any scale function $h$ satisfies $h(t)\le Ct$ for $t\in [0,1]$.
	Hence the Carleson constant of every set $\C\sm \cG_s(\mu)$ with respect to $\mu$ is
	bounded. Indeed, choose $k\ge 0$ with $2^{-k}<s$.
	Let $z\in \C\sm \cG_s(\mu)$. The closest point $p$ to $z$ in the discrete grid
	$2^{-k}\Z\times 2^{-k}\Z$ is the vertex of four dyadic squares from the generation
	$\cD_k$. If $Q$ is any of these squares, then $|Q|\le s$ and $10 Q^o$ contains $z$, which implies that
	$\mu(Q)\le h(|Q|) \le C|Q|$.
	This shows that the constant $\Carl(\C\sm \cG_s(\mu), \mu)$ is bounded for any $s$.
	Denote by $X_s$ and $Y_s$ the $x$- and $y$- projections of the set $\cG_s(\mu)$.
	They are open and their lengths tend to zero as $s\to 0$.
	Note also that $\{X_s\}$ and $\{Y_s\}$ are increasing families of sets.
	Define $\A_x$ to be the set of points in $\R$ that are
	density points of at least one of the sets $\R\sm X_s$.
	Define $\A_y$ similarly, using
	$Y_s$ in place of $X_s$.
	Then $\A_x$ and $\A_y$ are measurable, and their complements in $\R$
	have Lebesgue measure zero. For any point $x_0$ in
	$\A_x$, there is some $s>0$ such that
	$x_0\in \R\sm X_s$ and $x_0$ can be approximated
	both from above and from below by a sequence of points in
	the same set $\R\sm X_s$. The points of
	$\A_y$ have a similar property.
	
	This proves assertion 2).
	
	Any broken line $\ga$ as in Lemma is contained in $\C\sm \cG_s(\mu)$ for some sufficiently
	small $s$. In conclusion assertion 1) holds.
\end{proof}

\begin{proof}[Proof of Proposition~\ref{sum-dissect-measures}]
	\textit{Case 1: }
	Assume first that the hypotheses hold, and all scale functions coincide:
	$h_k=h$ for all $k$.
	We normalize the measures $\mu_k$, assuming $\mu_k(\C)=1$ for all $k$.
	We prove that there exist constants $c_k>0$ such that
	$\sum_k c_k\mu_k$ is a finite and $h$-dissectible measure.
	Put $c_k=\al_k^2$, where $\al_k>0$.
	
	First choose positive numbers $s_k\to 0$
	such that
	\beqn
	\label{estim-s-k}
	\Honecont\big(\cG_{s_k}(\mu_k)\big) \le 2^{-k}.
	\neqn
	Note that if a square $Q$ is $\al_k\mu_k$-heavy, then
	\[
	h(|Q|)< \al_k\mu_k(|Q|)\le \al_k.
	\]
	Choose $\al_k\in (0,1)$ so that these inequalities imply that $|Q|\le s_k$.
	We also assume
	\[
	\ga:=\sum_k\al_k\le 1.
	\]
	With this choice, the measure $\nu:=\sum_k c_k\mu_k$ turns out to be
	finite and $h$-dissectible. Indeed, fix any $\eps>0$. Find some $N$ such
	that $2^{-N}<\eps/2$.
	Since each $\mu_k$ is
	$h$-dissectible, we can choose $q>0$ so that $\Honecont\big(\cG_q(\mu_k)\big)< \eps/(2N)$ for
	$k=1, 2, \dots, N$.
	We assert that whenever $0<s<q$,
	\beqn
	\label{nu-diss}
	\cG_s(\nu)\subset \bigcup_{k=1}^N \cG_q(\mu_k)\cup \bigcup_{k=N+1}^\infty \cG_{s_k}(\mu_k).
	\neqn
	In view of \eqref{estim-s-k} we find:
	\[
	\Honecont\big(\cG_s(\nu)\big)<
	N\cdot \frac{\eps}{2N}+\sum_{k=N+1}^\infty 2^{-k}=\frac \eps 2 + 2^{-N}<\eps,
	\]
	and therefore, $\nu$ is dissectible, because $\eps$ was arbitrary.
	
	So it remains to check~\eqref{nu-diss}. In order to do so,
	take any dyadic square $Q$, which is $\nu$-heavy and satisfies
	$|Q|<s<q$.
	Then
	$\ga\nu(Q)\ge \ga h(|Q|)$, which we rewrite as
	\[
	\ga\sum_k \al_k^2 \mu_k(Q)\ge \sum_k \al_k h(|Q|).
	\]
	Hence for some $k$,
	\[
	\mu_k(Q)\ge \al_k\mu_k(Q) \ge \frac 1 \ga\, h(|Q|)\ge h(|Q|).
	\]
	So $Q$ is $\al_k\mu_k$-heavy (and $\mu_k$-heavy). By the choice of $\al_k$, this implies $|Q|\le s_k$.
	Hence
	\[
	10 Q^o \subset \cG_t(\mu_k),
	\]
	where $t=\min(s, s_k)<q$. This implies relation ~\eqref{nu-diss}.

	\medskip

	\textit{Case 2 (the general case): }
	Notice first that whenever $h$, $h_1$ are scale functions with
	$h_1 \le h$, then $Q$ $h$-heavy implies that $Q$ is
	$h_1$-heavy. Hence, if a measure $\nu$ is  $h_1$-dissectible, then it
	is also $h$-dissectible with respect to any scale function $h$
	such that $h_1 \le h$ on an interval $[0,\eps]$, $\eps>0$.
	
	As a consequence, we will infer the general case from
	the above Case 1 once we check
	the following claim:
	\textit{For any sequence of scale functions $\{h_k: k\ge 1\}$,
		there exists a scale function $h$ such that for any $k$,
		$h\ge h_k$ on some (nonempty) interval $(0, \eps_k]$.
	}
	
	This is simple. Namely, let $\{\be_k\}$
	be a sequence of positive numbers such that
	$\sum_k\beta_k<\infty$.
	Put
	\[
	h(t)=h_1(t)+\sum_{k\ge 2} g_k(t),
	\]
	where
	$g_k(t)=\min(h_k(t), \beta_k)$.
	Then $h(t)$ is finite for any $t\ge 0$.
	If the numbers $\beta_k>0$ tend sufficiently rapidly to zero, then
	$h$ satisfies the integral condition~\eqref{int-condn-h} and so it is a desired scale function.
\end{proof}

\section{PROOFS PERTAINING TO THE FUNCTIONAL MODEL}
\label{sect-proofs-model}

First we mention Cauchy-Pompeiu formula for functions belonging to the space
$\Mod(\Om)$: \textit{ If $\Om$ is admissible and $w\in \Mod(\Om)$,
	then
	\beqn
	\label{Cauchy-Pomp}
	(2\pi i)^{-1} C_{\pt\Om}(w)(z) -
	\pi^{-1} C(\bar\pt w)(z) =
	\begin{cases}
		w(z), & \text{if } z\in\Om \\
		0, & \text{if } z\notin\overline{\Om}
	\end{cases}
	\neqn
	(notice that $w|\pt\Om\in L^2(\pt\Om, |dz|)$).
}

Indeed, by the definition of $\Mod(\Om)$, it suffices to check
this formula for functions $w=b\in \Etwo(\Om)\otimes \C^m$ and for
functions of the form $w=C(x \odot \bar v\cdot \mu)$, where $v\in
L^2(\mu, H_0)$. In the first case, it is just Cauchy's formula.
The second case is also obvious because
$\bar\pt w = - \pi x \odot \bar v \cdot \mu$.

\begin{proof}[Proof of Proposition \ref{prop-traces}]
	Let $g=C (x\odot \bar v\cdot\mu) + b$ be a function in $\Mod(\Om)$, as
	in the Definition~\ref{def-model-space} of the model space
	and let $\ga$ be a curve meeting the hypotheses. Suppose first that
	$\ga\subset\Om$.
	
	If $x$ is bounded, then by (D$'$) and Lemma~\ref{L-infty-lemma},
	one can define $g$ on the set $\Om\sm \cap_s\cG_s(\mu)$; it will be
	continuous on $\Om\sm \cG_s(\mu)$ for any $s>0$. Since $\ga$ does not
	touch some set $\cG_{s_0}(\mu)$, this permits one to define
	$g|\ga$ as a restriction.
	
	Now Lemma~\ref{lem-Cauchy-mu} implies that the map
	$g\to g|\ga$ extends by continuity to functions $g\in \Mod(\Om)$ as
	above, corresponding to arbitrary elements $x\in L^2(\mu, H_0)$.
	
	We adopt this extension as definition of the trace.
	
	By the same Lemma~\ref{lem-Cauchy-mu}, the map
	$g\to g|\ga$ is also well-defined and bounded if $\ga$ is
	a subarc of $\pt\Om$.
\end{proof}

\begin{proof}[Proof of Proposition \ref{prop-spectrum-T}]
	(1) The definition of $\cH_0$ implies that
	the essential spectrum $\si_{ess}(N|\cH_0)$, and hence
	also $\si_{ess}(T_0)$, coincide with
	$\Der\supp\mu$.
	
	(2) Fix a point $\la$, which is not in the support of $\mu$. By the above,
	$\la\in\si(T_0)$ if and only if $\la$ is an eigenvalue of $T_0$. This rewrites as
	$(z-\la)f(z)=-u^t c$, where $f\in\cH_0$ ($f\ne 0$) is the corresponding eigenfunction
	and $c=\langle f, v\rangle$. It is easy to see that these two equations are solvable
	if and only if $\Psi(\la) c=0$ has a nonzero solution.
\end{proof}

\begin{proof}[Proof of Proposition \ref{prop-closed-ss}]
	Suppose $\pt\Om$ admissible.

	(a) First we check that $\Psi \big(\Etwo(\Om)\otimes \C^m\big)$ is
	contained in $\Mod(\Om)$. Take any $a\in \Etwo(\Om)\otimes \C^m$.
	We need to verify that $g:=\Psi a\in \Mod(\Om)$. Note the identity
	\beqn
	\label{bar pt g}
	\bar\pt g= - \pi (\bar v\odot u^t)\cdot a= - \pi \bar v \odot (u^t
	\cdot a) \quad \text{in }\Om.
	\neqn
	Since Carleson's embedding operator $J_\Om$
	is bounded on $\Etwo(\Om)$ and $u$ is a bounded function,
	$x:=-(u^t \cdot a)\chi_\Om$ belongs to $L^2(\mu)$. By \eqref{bar pt g},
	the function
	\[
	b:= g - C (\bar v \odot x\cdot \mu)
	\]
	is analytic in $\Om$.
	Let $\Om_n$ be the domains that correspond to $\Om$,
	whose existence is asserted in the Definition~\ref{def-adm-domain} of
	an admissible Jordan curve. Since Carleson's constants of $\pt\Om_n$ with respect to $\mu$ are
	uniformly bounded and $\Psi$ is bounded on the union of these curves
	(see Lemma~\ref{L-infty-lemma}), we deduce
	\[
	\sup_n \|C(\bar v \odot x\cdot \mu)\|_{L^2(\pt\Om_n)}<\infty
	\quad \text{and} \quad
	\sup_n \|g\|_{L^2(\pt\Om_n)}<\infty.
	\]
	Therefore the integrals $\int_{\pt\Om_n}|b|^2\, |dz|$ are
	uniformly bounded, which means that $b$ belongs to $E^2(\Om)$. Hence
	$g= b + C(\bar v \odot x \cdot \mu)$ is in $\Mod(\Om)$.

	The same arguments imply the boundedness of the map $a \to \Psi a$, acting on
	$\Etwo(\Om)\otimes \C^m$ and with values on $\Mod(\Om)$.
	
	(b) We verify now the last statement referring to the trace of $\Psi a$ on $\pt\Om$.
	Let $C_a({\overline\Om})$ stand for the space of functions, analytic in $\Om$ which
	can be continuously extended to the boundary.
	Consider first the case $a\in C_a({\overline\Om})\otimes \C^m$
	and define $x,b$ as above.
	Then $x\in L^\infty(\mu)$, which implies that $C (\bar v \odot x\cdot \mu)$ is well-defined at any point of
	$\pt\Om$ and is continuous on this curve. Hence $b$ is also continuous on $\pt\Om$ and so
	$b\in C_a({\overline\Om})\otimes \C^m$.
	Thus equality $\Psi a = C (\bar v \odot x\cdot \mu) +b$ holds pointwisely on
	$\pt\Om$, which implies that the trace of $\Psi a$ on $\pt\Om$ equals
	$\Psi\cdot a|\pt\Om$.
	
	The case of a general element $a$ in $\Etwo(\Om)\otimes \C^m$ is obtained by approximating $a$
	in the norm of $\Etwo(\Om)\otimes \C^m$ by a sequence
	of functions $a_n$ in $C_a({\overline\Om})\otimes \C^m$ and
	by invoking the boundedness of the trace map $w\mapsto w|\pt\Om$, $w\in \Mod(\Om)$.

	(c)
	To prove that $\Psi \big(\Etwo(\Om)\otimes \C^m\big)$ is a closed subspace
	of $\Mod(\Om)$, we use expression
	\eqref{equiv-norm-mod-space}, which gives an equivalent norm on $\Mod(\Om)$.
	Since $\|\Psi^{-1}\|$ is uniformly bounded on
	$\pt\Om$, it follows that $\|\Psi a\|_{\Mod(\Om)}\ge \de \|a\|_{\Etwo}$ for
	any $a\in \Etwo(\Om)\otimes\C^m$, where $\de>0$. Hence the
	image of the map $a\in \Etwo(\Om)\otimes\C^m\mapsto \Psi a$ is closed in $\Mod(\Om)$.
\end{proof}

\begin{proof}[Proof of Lemma \ref{lemma-multiplication-oper}]
	Let $g\mapsto x(g)$ be the map defined in~\eqref{bded-maps-mod-space}.
	Since $x(g)$ is obtained via the Cauchy-Riemann operator $\bar\pt g$, it follows that
	$x(a g)= a x(g)$ for $g\in \Mod(\Om)$ and  $a$ analytic in a
	neighborhood of $\overline\Om$. Applying this observation to $a(z)=z$ and using the equivalent
	norm~\eqref{equiv-norm-mod-space} on the model space, we deduce
	the boundedness of the operator $g\mapsto zg$ on $\Mod(\Om)$. For the same reasons,
	for any $\la\in \Om^c$, the multiplication operator
	$g\mapsto (z-\la)^{-1} g$ provides a bounded inverse to the operator
	$g\mapsto (z-\la) g$.
\end{proof}

\begin{proof}[Proof of Theorem \ref{thm-model}]
	We can repeat the arguments in the proof of
	Theorem~1 in \cite{Y1} (see \cite{Y1}, p. 66).
	The intertwining property follows from ~\eqref{intertw-W0-new}.
	It remains to check that, given
	$w_0$ in $\Mod_B$, there exist a unique $x$ in $L^2(\mu,H_0)$
	and a unique $a\in E^2(\cB)\otimes \C^m$ satisfying
	\beqn
	\label{eqn-toe}
	(W_{0, \cB} x)(z)= w_0(z)+\Psi(z)a(z), \quad z\in \cB.
	\neqn
	We appeal to a class of Toeplitz type operators.
	If $F$ is matrix-valued function in $L^\infty(\pt\cB, L(\C^m))$,
	then the Toeplitz operator $\toe_F$ acts on $E^2(\cB)\otimes \C^m$ by
	\[
	\toe_F f(z):= C_{\pt\cB}(Ff)(z), \quad z\in \cB.
	\]
	It is bounded. It is easy to verify that
	$\toe_G\toe_F= \toe_{GF}$ whenever
	$F, G\in H^\infty(\cB^c, L(\C^m))$. In particular, $\toe_{\Psi}$ is invertible, and
	$\toe_{\Psi}^{-1}=\toe_{\Psi^{-1}}$.

	Notice that, for a function $g$ in $\Mod(\cB)$, one has
	\beqn
	g\in W_{0,\cB} L^2(\mu|\cB, H_0) \Leftrightarrow
	C_{\pt\cB} (g) =0 \quad \text{in }\cB.
	\neqn
	Therefore \eqref{eqn-toe} is solvable with respect to $x$ if and only if
	\[
	a=-\toe_{\Psi^{-1}}\big(C_{\pt\cB}(w_0)\big)\big|_\cB.
	\]
	If we define $a$ by this formula, by Proposition~\ref{prop-closed-ss}, $\Psi a$ and hence
	$w:=w_0+\Psi a$ belong to $\Mod(\cB)$. Therefore $a$ and $x$ are determined by
	\eqref{eqn-toe}.
\end{proof}

\begin{proof}[Proof of the Glueing lemma \ref{lemma-Glueing}]
	Suppose $w_1$, $w_2$ satisfy the hypotheses
	and define a function $w$ on
	$\Om_1\cup \Om_2$ by $w\big|_{\Om_j}=w_j$.
	By summing formulas~\eqref{Cauchy-Pomp}
	applied to $\Om_1$, $\Om_2$, we infer
	\[
	w= (2\pi i)^{-1} C_{\pt\Om}(w) - \pi^{-1} C\big((\bar\pt w_1)\chi_{\Om_1}+(\bar\pt w_2)\chi_{\Om_2}\big)
	\]
	on $\Om_1\cup \Om_2$. The right hand part is a function in
	$\Mod_\Om$.
\end{proof}

We remark that the same lemma holds for finitely many bordering domains, instead of two
of them.

\begin{proof}[Proof of Theorem \ref{thm-direct-sum}]
	%
	%
	We can follow the lines of the proof of Lemma~6 in \cite{Y1}. First we
	prove the assertion for the operator $T_0=T|\cH_0$, whose functional
	model is given by Theorem~\ref{thm-model}.
	Consider the restriction maps $\cJ_k w=w|\Om_k\in \Mod(\Om_k)$, $w\in \Mod(\Om)$.
	It follows from Proposition~\ref{prop-traces} and
	the expression~\eqref{equiv-norm-mod-space} for the norm in the model spaces that
	these are bounded linear operators.
	They induce operators
	$\wh \cJ_k: \cQ(\Om)\to \cQ(\Om_k)$, which are well-defined and bounded.
	We assert that
	$\wh\cJ=(\wh\cJ_1, \dots, \wh\cJ_N): \cQ(\Om)\to \oplus_k \cQ(\Om_k)$ is an isomorphism.
	Let $w\in \Mod(\Om)$, and let
	$\hat w$ be the corresponding element (class of equivalence) in $\cQ(\Om)$.
	If $\wh\cJ \hat w=0$, then $w|\Om_k=\Psi a_k$ for some functions $a_k\in E^2(\Om_k)$, $k=1, \dots, N$.
	It is easy to check that for any
	$g\in\Mod(\Om)$ and any domains $\Om_j$, $\Om_k$, bordering
	by an arc $\ga$, the trace of $g|\Om_j$ on $\ga$ equals the
	trace of $g$ on $\ga$. Hence,
	by the last statement of Proposition~\ref{prop-closed-ss},
	$a_k=a_j$ on $\pt\Om_k\cap \pt\Om_j$, $1\le k< j\le N$.
	By applying the Cauchy integral representation of $E^2$ functions,
	one finds there is some $a\in E^2(\Om)\otimes \C^m$ such that
	$a_k=a|_{\Om_k}$ for all $k$. This shows that $\ker\wh\cJ=0$.
	
	The fact that $\wh\cJ$ is also onto is shown in the same way as in 
	the proof
	of Lemma~6 in \cite{Y1}, and we leave the details to the reader.
	
	In the general case, the
	assertion follows, because
	$T|\cH\ominus \cH_0$ is normal and
	$T|\cH_0$ has the same perturbation determinant as $T$.
\end{proof}

\section{THE EXISTENCE OF INVARIANT SUBSPACES}

The aim of this section is to prove the following result.

\begin{theorem}\label{thm_hyperinf}
	Suppose the linear operator $T$ is given by~\eqref{perturb-T} and satisfies conditions (D) and (K).
	If either $\Der\supp\mu$ is not connected or
	there exists a domain $G$, whose boundary is a Lipschitz Jordan curve,
	such that $G\cap \Der\supp\mu=\emptyset$ and
	the intersection of $\supp\mu$ with the boundary of $G$
	contains an arc,
	then $T$ has a nontrivial
	invariant subspace.
\end{theorem}

The proof of this result will rely on the following known fact.

\begin{thmA}[Privalov's uniqueness theorem, see \cite{Privalov-book}, Ch. IV, \S2.6]
	If a function $f(z)$, meromorphic on a domain, bounded by a rectifiable Jordan curve
	$\Ga$, has angular limit values equal to zero on a subset $E$ of $\Ga$ with
	$\cH^1(E)>0$, then $f$ is identically zero.
\end{thmA}

Recall also F. and M. Riesz theorem \cite{Riesz, Privalov-Saratov}:
if $\Om$ is any Jordan domain, bounded by a rectifiable curve,
then the Hausdorff measure $\cH^1$ and the harmonic measure
are mutually absolutely continuous on $\pt\Om$.

\begin{proof}[Proof of Theorem~\ref{thm_hyperinf}]
	
	We can assume (D${}'$), (B) and also that $\cH=\cH_0$.
	We start by the following
	reduction. Consider the open set
	\[
	\Om:=\C\sm\Der\supp\mu.
	\]
	Note that $\Psi$ and $\psi$ are meromorphic functions on $\Om$.
	If $\psi$ vanishes on one of its connected components,
	then, by part (2) of Proposition~\ref{prop-spectrum-T},
	$T$ has an eigenvalue, and therefore has an
	invariant subspace. So, from now on, we will assume that
	$\psi$ does not vanish identically on any of the connected components of $\Om$.
	
	\textbf{Case 1:} The (compact) set $\Der\supp\mu$ is disconnected.
	Then it decomposes into a disjoint union of two non-empty closed and relatively open subsets, say, $F_1$ and $F_2$.
	Since the zeros and the poles of $\psi$ form discrete subsets of the complement
	of $F_1\cup F_2$, one can find open disjoint sets $\cO_1\supset F_1$ and
	$\cO_2\supset F_2$, whose boundaries are finite unions of rectifiable Jordan curves,
	do not intersect with the set $F_1\cup F_2$ and do not contain neither zeros nor poles of $\psi$.
	Note that $\pt\cO_1$ and $\pt\cO_2$ are contained in $\C\sm\si(T)$, whereas
	$F_1$ and $F_2$ are contained in $\si(T)$.
	It follows that the corresponding Riesz projection
	$P_{\cO_1}$ is a non-trivial idempotent, commuting with $T$. The range of
	$P_{\cO_1}$ is a nontrivial invariant subspace of $T$.

	\textbf{Case 2:} There exists a domain $G$, meeting the hypotheses of the theorem.
	By passing to a smaller domain, we may assume that $\supp\mu\cap \pt G$
	has the form
	\[
	\ga=\{x+if(x):\; x\in J\},
	\]
	where $J\subset \R$ is a finite closed
	interval and $f$ is a Lipschitz function.
	Since $\psi$ is meromorphic in $G$ and is not identically zero in this domain,
	Privalov's uniqueness theorem implies $\psi(\la)\ne 0$ for a.e. $\la\in \ga$.
	We can also assume that $f$ is non-constant on $J$.
	
	There is a compact subset $F$ of  $f(J)\cap \A_y$ of positive measure.
	Since $f$ maps sets of zero measure to sets of zero measure
	and $f(f^{-1}(F))=F$, the preimage $f^{-1}(F)$ has positive measure.
	Hence one can choose $z_0=x_0+i y_0\in \ga$ so that
	$x_0\in J^o\cap \A_x\cap f^{-1}(\A_y)$ and $\psi(z_0)\ne 0$.
	Then $y_0\in \A_y$.

	Part 2) of Lemma~\ref{lem-broken-line} implies that there
	exist sequences $\{x_n^-\}, \{x_n^+\}$ tending to $x_0$ and
	$\{y_n^-\}, \{y_n^+\}$ tending to $y_0$ such that
	$x_n^-<x_0<x_n^+$ and $y_n^-<y_0<y_n^+$ for all $n$.
	Moreover, there is $s>0$ such that for all $n$,
	$x_n^-,x_n^+\in \R\sm\Re \cG_s(\mu)$
	and
	$y_n^-,y_n^+\in \R\sm\Im \cG_s(\mu)$.
	
	Consider a rectangle $R=[x_n^-,x_n^+]\times [y_n^-,y_n^+]$. If
	$n$ is sufficiently large, then
	$|\psi|>\de>0$ on $ R \sm \cG_s(\mu)$ and
	$\ga$ is not contained in $R$.
	
	Notice also that the boundary of $R$ is contained in $\C\sm\cG_s(\mu)$.
	Therefore $R^o$ is an admissible domain.
	
	By Theorem~\ref{thm-direct-sum}, $T$ has invariant subspaces $L$ and $M$ such that
	$\cH_0=L\dotplus M$, $\si(T|L) \subset \si(T)\cap R$ and $\si(T|M)\subset\si(T)\sm R^o$.
	Notice that $\si(T)=\si(T|L)\cup \si(T|M)$. Therefore each one of the spectra
	$\si(T|L)$ and $\si(T|M)$ contains a nontrivial subarc of $\ga$.
	Hence $L\ne 0$ and $M\ne 0$, so that $L$ is a nontrivial invariant subspace of $T$.
\end{proof}

\section{BISHOP PROPERTIES ON THE MODEL SPACE AND DECOMPOSABILITY}
\label{sec-Bishop-decomp-diss}

A landmark contribution to axiomatic spectral theory is Bishop's 1959 article \cite{B}. Inspired by generalized spectral decompositions
of linear and bounded Banach space operators lacking a spectral measure, he has identified four different behaviors
of the resolvent of the dual
which imply the existence of invariant subspaces localizing the spectrum. Soon afterwards Foia\c{s} has
isolated in 1963  the concept of decomposable operator \cite{Fo};
this class of linear operators and Bishop's properties have simplified and unified conceptually
many lines of research in spectral analysis and produced over the years far reaching applications. For an early account
of the theory we refer to \cite{CF, DS} as for recent developments, including multivariate generalizations, see \cite{EP}.

A linear and bounded operator $T$ acting on a Banach space $X$ is called {\it decomposable},
if for every finite open cover of its spectrum
\beqn
\label{open-cover}
\sigma(T) \subset \cup_{j=1}^n U_j,
\neqn
there are $T$-invariant subspaces $X_j \subset X, \ \ 1 \leq j \leq n,$ with the properties
\beqn
\label{decomp-X}
X = X_1 + X_2 + \ldots + X_n,
\neqn
and
$$
\sigma(T|_{X_j}) \subset U_j, \ \ 1 \leq j \leq n.
$$
Note that the above decomposition is not a direct sum, nor there are bounded linear projections onto its terms.
Later on it was proved that only open covers with two sets suffice for the decomposability condition.

Examples are all classes
of operators possessing a spectral measure, and beyond, for instance operators admitting a functional calculus with
smooth functions.

A decomposable operator $T \in L(X)$ has the {\it single valued extension property}: for every open set $U \subset \C$ and any vector valued analytic function $f \in {\mathcal O}(U,X)$,
$(z I - T)f(z) = 0, \ \ z \in U,$ implies $f = 0$. As soon as an operator $T \in L(X)$ has the
single valued extension property, one can speak without ambiguity about the localized spectrum
$\sigma_x(T)$ of $T$ with respect to a vector $x \in X$, defined as the smallest closed subset
of the complex plane allowing the localized resolvent $(z I - T)^{-1} x$ to have an analytic
continuation on its complement \cite{CF}.

One step further, the operator $T \in L(X)$ satisfies {\it Bishop's property $(\beta)$} if the map
$$ (z I - T) : {\mathcal O}(U,X) \longrightarrow {\mathcal O}(U,X)$$
is one to one with closed range for every open set $U \subset \C$.
Here $\cO(U,X)=\cO(U) \otimes X$ stands for
the Fr\'echet space of all analytic $X$-valued functions on $U$.
Obviously it is sufficient to check this condition on
open disks $U$.

A decomposable operator possesses property $(\beta)$ \cite{B,CF}. A more recent theorem
due to Albrecht and Eschmeier \cite{AE}
completes Bishop's visionary program by stating that
$T$ is decomposable if and only if $T$ and its dual $T^\ast$ both have property $(\beta)$. All these results have
an analog in the case of commuting tuples of operators.
In that context the analytic sheaf model
$$ {\mathcal F}(U) = {\mathcal O}(U,X)/ (z I - T) {\mathcal O}(U,X)$$
is prevalent, opening the gate to homological algebra techniques \cite{EP}.

Prompted by the spectral behavior our main theorem reveals, we propose the following
more restrictive variation of the decomposability property.
If $A\subset B\subset \C$, we denote by $\pt(A;B)$ the relative boundary of $A$ in $B$.

\begin{definition}
	Let $X$ be a Banach space.
	We will say that
	an operator $T\in L(X)$ is \emph{dissectible}
	if
	for any open cover
	\[
	\si(T) \subset \cup_{j=1}^n U_j
	\]
	there are closed sets $F_j \subset U_j$ such that
	$\si(T) = \cup_{j=1}^n F_j$, $F_j\cap F_k=\pt (F_j;\si(T))\cap \pt (F_k;\si(T))$
	for all $j\ne k$, and there are $T$-invariant
	subspaces $X_j \subset X, \ \ 1 \leq j \leq n,$ such that
	\[
	\sigma(T|_{X_j}) \subset F_j, \ \ 1 \leq j \leq n
	\]
	and a direct sum decomposition holds:
	\beqn
	\label{decomp-X2}
	X = X_1 \dotplus  X_2 \dotplus  \ldots \dotplus X_n.
	\neqn
\end{definition}

It is clear that this notion is stronger than decomposability.
Any normal operator and any compact operator are dissectible.
On the other hand, plenty of decomposable operators are not dissectible.

To fix ideas we discuss a simple case.
Let  $\Y\subset \C$ be any connected compact set,
which has more than one point.
Let
$X = C(\Y)$ be the space of continuous functions on $\Y$.
The multiplication operator $Tf(z)=zf(z)$ is decomposable but not dissectible.

Indeed,
it is easy to see that  $\si(T)=\Y$.
Take any open cover $\Y\subset U_1\cup U_2$ of $\si(T)$ such that neither $U_1$ nor $U_2$ covers $\Y$.
Notice that if $F_1$, $F_2$ correspond to this open cover, then
$f|\pt(F_j; \Y)=0$ for all $f$ in $X_j$. Since $\Y=F_1 \cup F_2$
is connected and $F_j\ne \emptyset$, there is a point $w$ in
$F_1\cap F_2=\pt (F_1; \Y)\cup \pt (F_2; \Y)$. By~\eqref{decomp-X2},
any function in $X$ vanishes at $w$, a contradiction.

To increase generality, take now $X$ to be a continuously embedded Banach space into $C(\Y)$, so that
$T=M_z$ on $X$ is decomposable and $\sigma(T) = \Y$
(such as a Sobolev space). The argument above adapts and
implies that $T$ is not dissectible.

We observe that Theorem \ref{thm-direct-sum} permits us to prove the following fact.

\begin{theorem}
	\label{corr-thm-direct-sums}
	Assume that conditions (D) and (K) are satisfied.
	If, moreover, for any $s>0$,
	the
	one-dimensional Hausdorff measure of the
	set
	\[
	\{z\in \C\sm\cG_s(\mu): \psi(z)=0\}
	\]
	equals to zero, then $T$ is dissectible.
\end{theorem}

We recall that, by Lemma~\ref{L-infty-lemma},
$\psi$ is continuous on $\C\sm\cG_s(\mu)$ for any $s>0$;
the sets $\cG_s(\mu)$ have been defined in~\eqref{eq3n}.

In the proof, we use the following notation.
Given some $z=x+iy$ and some radius $r>0$, we set
\[
\Bxm(z,r)=\{w\in B(z,r): \Re w < x \},
\quad
\Bxp(z,r)=\{w\in B(z,r): \Re w> x \};
\]
to be the left and the right half of the disc $B(z,r)$.
These two sets are open.

Let $\Y\subset\C$, and let $z\in \Y$. We will say that
\textit{$z$ is $r$-accessible from the right in $\Y$} if
the intersection $\Y\cap \Bxp(z,r)$ is not empty.

We say that a point $z$ \textit{is accessible from the right in $\Y$} if it
$r$-accessible from the right in $\Y$ for any positive $r$. Equivalently,
there should exist a sequence $\{w_k\}$ of points of $\Y$ that tends to $z$ and
satisfies $\Re w_k>\Re z$ for all $k$.

Note that $z$ is inaccessible from the right
whenever it not $r$-accessible from the right for some $r>0$.
We define similarly the accessibility from the left, from above and from below.

\begin{lemma}
	For any bounded subset $\Y$ of the complex plane, the set of its
	points, inaccessible from the right in $\Y$, is contained in a countable
	union of vertical (straight) lines.
\end{lemma}

\begin{proof}
	A point of $\Y$ is inaccessible from the right (in $\Y$) if and only if
	it is $1/k$-inaccessible from the right
	for some $k\in\N$. So it will be enough to check that, say,
	for any $r\in (0,2)$, the set of points of $\Y$, $r$-inaccessible from
	the right, can be covered by finitely many vertical lines.
	
	Fix some $r\in (0,2)$ and assume that $\Y\subset B(z_0,R)$ for some $z_0$
	and some $R>0$. We will prove that the above set of points can be covered
	by no more than $[4(R+1)^2/r^2]$ vertical lines, where
	$[t]$ stands for the entire part of $t$.
	
	Suppose it is false. Then there exist $N>4(R+1)^2/r^2$
	points $z_j\in \Y$  that are $r$-inaccessible from the right
	and all have distinct real parts. By comparing areas, we get that
	for some $k\ne \ell$, the discs $B(z_k,r/2)$ and $B(z_\ell,r/2)$
	have to intersect. Hence $|z_k-z_\ell|<r$. Let, for instance,
	$\Re z_k<\Re z_\ell$. Then $z_\ell\in \Bxp(z_k,r)\cap \Y$.
	Therefore $z_k$ is $r$-accessible from the right, a contradiction.
	This finishes the proof.
\end{proof}

\begin{proof}[Proof of Theorem~\ref{corr-thm-direct-sums}]
	By applying last Lemma four times to the set $\Y=\si(T)$,
	we find that in Lemma~\ref{lem-broken-line} the sets $\A_x, \A_y\subset \R$
	can be chosen so that they satisfy the following additional properties:
	
	\begin{enumerate}
		\item[(P1)]  Any point $z\in \si(T)$ such that $\Re z\in \A_x$ is accessible
		from the right and from the left in $\si(T)$;
		\item[(P2)]  Any point $z\in \si(T)$ such that $\Im z\in \A_y$ is accessible
		from above and from below in $\si(T)$.
		\item[(P3)]  The lines $\{\Re z=x_0\}$, where $x_0\in \A_x$, and
		$\{\Im z=y_0\}$, where $y_0\in \A_y$, are contained in $\{z\in \cG_s: \psi(z)\ne 0\}$
		for some positive $s$.
	\end{enumerate}
	
	We also assume that whenever $a$ is an isolated point
	of $\si(T)$, $\Re a\notin \A_x$ and $\Im a\notin \A_y$.
	This is achieved by quitting some countable subsets from
	$\A_x$ and $\A_y$, once again.
	
	Suppose an open cover \eqref{open-cover}
	of $\si(T)$ is given.
	Draw finitely many vertical lines $\{\Re z= x_j\}$ ($1\le x_j\le N$), where
	$x_j\in \A_x$, and finitely many horizontal lines $\{\Im z= y_j\}$ ($1\le y_j\le M$),
	where $y_j\in \A_y$. Let us assume that
	the finite sequences $\{x_j\}$, $\{y_j\}$ are increasing and that
	the open rectangle $(x_1,x_N)\times (y_1,y_M)$
	contains $\si(T)$.
	Put $R_{jk}=[x_j,x_{j+1}]\times [y_k,y_{k+1}]$.
	By Lebesgue Lemma, we can also assume that
	the lines that were drawn are so close to each other that
	for each pair $(j,k)$, there is an
	index $\hat m(j,k)$ such that $R_{jk}\subset U_{\hat m(j,k)}$.
	Fix these numbers $\hat m(j,k)$, and set
	\[
	\wt R_m=\cup \{R_{jk}: \; \hat m(j,k)=m\}.
	\]
	Then $\wt R_m$ have disjoint interiors and $\wt R_m\subset U_m$.
	
	The desired sets $F_m$ will be defined as
	\beqn
	\label{Fm}
	F_m=\si(T)\cap R_m,
	\neqn
	where $R_m$ are certain modifications of $\wt R_m$. These
	modifications are performed as follows.
	By a vertex, we mean a point of the form
	$(x_j, y_k)$, which is on the boundary of one of polygonal sets
	$\wt R_m$. We say that a vertex
	$p$ is special if it is a limit point of $\si(T)\cap \wt R_m$ only
	for one index $m=m(p)$.  (For any vertex $p\in \si(T)$, such index $m$ should exist.)
	
	For any special vertex $p$, choose a small closed rectangle $\rho=\rho(p)$,
	whose vertices lie on $\A_x+i\A_y$, that contains $p$ in its interior and
	does not touch any drawn line, except those
	that pass through $p$. We also require that
	$\rho\subset U_{m(p)}$ and that $\rho\cap \si(T)=\wt R_{m(p)}\cap \si(T)$.
	
	Notice that $p$ cannot be a limit point of $\si(T)\cap \pt \wt R_m$
	(otherwise, as it is easy to see from (P1) and (P2) above, $p$
	would be also a limit point
	of $\si(T)\cap \wt R_t$ for some $t\ne m$). Therefore
	$\rho$ can be chosen so that
	\beqn
	\label{rho}
	\rho\cap (\si(T)\sm\{p\})=\wt R_m^o \cap\rho\cap (\si(T)\sm\{p\}).
	\neqn
	
	We make replacements
	\[
	\wt R_{m(p)}\mapsto \wt R_{m(p)}\cup \rho(p)
	\]
	and
	\[
	\wt R_{t}\mapsto \wt R_{t}\sm \rho(p)^o
	\]
	whenever $t\ne m$ and $p$ lies on the boundary of $\wt R_t$.
	After making these replacements for all special vertex points $p$, we obtain
	modified sets $R_m$, such that special vertex points are no longer their vertices.
	(We may assume that the rectangles $\rho(p)$, corresponding to different
	special vertex points $p$, are disjoint, so that the result does not depend on the order
	of these replacements.)
	Observe that $R_m$ have disjoint interiors and $R_m\subset U_m$.
	
	Define the sets $F_m$ by~\eqref{Fm}.
	By~\eqref{rho}, all vertices of the sets $R_m$
	that were not among the vertices of $\wt R_m$ do not
	belong to $\si(T)$.

	By (P3) and Theorem~\ref{thm-direct-sum}, there is a decomposition
	$T=T_1\dotplus \dots \dotplus T_N$, where $\si(T_m)\subset F_m$.
	
	It remains to check the condition concerning the boundaries of $F_m$.
	Let $w\in F_m\cap F_t$, with $m\ne t$.
	Then $w$ belongs to the boundaries of $R_m$ and of $R_t$.
	If $w$ is  not a vertex point, then it follows from
	(P1) and (P2) that $w$ is a limit of points
	in $\si(T)\sm F_m$, and so $w\in \pt(F_m; \si(T))$. Similarly,
	$w\in \pt(F_t;\si(T))$.
	
	Finally, let $w$ be a vertex point.
	Since $w\in \si(T)$, it is not a special point.
	Therefore there are indices $k\ne \ell$
	such that $w$ is a limit point of both sets
	$F_k$ and $F_\ell$. Either $k\ne m$ or $\ell\ne m$; assume
	for instance that $k\ne m$. Then $w$ is a limit point
	of the set $\si(T)\cap R_k^o$, which
	does not intersect $F_m$, and therefore
	$w\in \pt(F_m; \si(T))$. Similarly,
	$w\in \pt(F_t; \si(T))$.
	
	In conclusion,
	$F_m\cap F_t=\pt(F_m; \si(T))\cap \pt(F_t; \si(T))$.
\end{proof}

It turns out that a sufficient condition for decomposability of a totally different nature
can be proved.

Indeed, define
\[
\tilde \psi(z)=\|\Psi(z)^{-1}\|^{-1}
\]
(the right hand part is understood as zero if the matrix $\Psi(z)$ is not invertible).
Notice that for a rank one perturbation, when $m=1$, $\tilde \psi=\psi=\Psi$. We also remark that $|\tpsi(z)|\le
\|\Psi(z)\|$ for a.e. $z$, so that $\tilde \psi$ is a nonnegative locally $L^1$ function.
The set of zeros of $\psi$ and of $\tpsi$ on each set $\C\sm\cG_s(\mu)$ coincide.

\begin{definition}
	We say that  $\tpsi$
	\emph{has no deep zeros}
	if for any bounded domain $\cD$ and any its compact subset $K$, there is a constant
	$C(\cD, K, \tpsi)$ such that the estimate
	\beqn
	\label{eq-deep-zeros}
	\sup_K |f|\le C(\cD, K, \tpsi) \int_{\cD} |\tilde \psi| |f|
	\neqn
	holds for any function $f$, holomorphic in $\cD$.
\end{definition}

\begin{theorem}
	\label{thm-beta-decomp}
	If the function $\tilde \psi$ has no deep zeros, then $T$ has property~$(\beta)$.
\end{theorem}

In virtue of a result proved by Domar \cite{Do}, we derive a
more tangible criterion.

First we define an auxiliary function $F^*$. Choose a disc $B(0,R)$, containing the spectrum of $T$, and let
$F^*$ be the decreasing rearrangement of $\log\log(e+\tilde \psi^{-1})|B(0,R)$.
That is, $F^*$ is a decreasing non-negative function on $[0, \pi R^2]$ such that, for any
$s>0$, the length of the interval $\{t: F^*(t)\ge s\}$ is equal to the area measure of the set
$\{z\in B(0,R): \log\log(e+\tilde \psi(z)^{-1})\ge s\}$.

\begin{theorem}
	\label{thm-beta-decomp-domar}
	As usual, assume (D${}'$), (B) and (K).
	If $\Psi$ is a H\"older-$\alpha$ function for some $\alpha\in (0,1]$ and
	\beqn
	\label{eq-cond-domar}
	\int_0^\eps (t^{-1} F^*(t))^{1/2}<\infty
	\neqn
	for some positive $\eps$, then $T$ has property $(\beta)$.
\end{theorem}

Theorem \ref{thm-beta-decomp-domar} has a simple corollary concerning decomposability.

\begin{corollary}
	\label{cor-decomp}
	Suppose the conditions (D${}'$) and (B). Suppose that for $\mu$-a.e. $z$, the linear span of the
	vectors $u_1(z), \dots, u_m(z)$ coincides with the linear span of
	the vectors $v_1(z), \dots, v_m(z)$.
	If $\Psi$ is a H\"older-$\alpha$ function for some $\alpha\in (0,1]$ and~\eqref{eq-cond-domar}
	holds, then $T$ is decomposable.
\end{corollary}

Indeed, one obtains the representation of $T^*$ as a perturbation of $N^*$
by passing to conjugates in~\eqref{perturb-T}. The hypotheses
on $u_j$ and $v_j$ imply that both representations
of $T$ and of $T^*$ satisfy condition (K).
Moreover, in this case the corresponding perturbation matrices $\Psi_T$ and
$\Psi_{T^*}$ are of size $m$ and satisfy $\Psi_{T^*}(z)=\Psi_T(\bar z)^*$.
(If only (K) is assumed, one has to add new ``fake'' terms to the representation of $T^*$, which
makes the size of $\Psi_{T^*}$ greater than $m$.)
Hence $\tpsi_{T^*}(z)=\tpsi_{T}(\bar z)$.
So we derive property ($\beta$) for both $T$ and $T^*$, and therefore $T$ is decomposable.

\begin{example}
	The decomposability can fail if $\psi$ is a smooth function that vanishes on an
	smooth arc $\de$ and decays very rapidly when approaching to this arc.
	This was (rather briefly) explained in~\cite{Y1} at the end of Section~6.
	We reproduce the argument here with a few more details.
	Assume that $\mu$ is absolutely continuous with respect to area measure and $m=1$.
	Assume also that $\cH=\cH_0$.
	Notice the following simple fact: if $S\in L(X)$ is a Hilbert space operator
	and $W: K\to C(\de)$ is its diagonalization, that is,
	$(WS x)(z)=z (Wx)(z)$ for all $x\in X$, then
	$(W x)|\de\sm \si(S)=0$, $x\in X$. Indeed, it follows that
	the function $\la\mapsto (\cdot-\la)^{-1}Wx\in C(\de)$ extends analytically
	from $\C\sm(\si(S)\cup \de)$ to $\C\sm\si(S)$.

	Fix a domain $\cB\supset \si(T)$.
	If $\de\subset\si(T)$ and $\psi$ and $\bar\pt\psi$ decay sufficiently
	fast when approaching $\de$, then, as explained in~\cite{Y1},
	the operator $Wf:=f|\de$ is a well-defined diagonalization operator
	from the model space $\cQ(\cB)$  to
	a space of quasianalytic functions on $\de$.
	Take now any open cover $\si(T)\subset U_1\cup U_2$ such that
	each of the sets $\de\sm U_1$ and $\de\sm U_2$ contains a  subarc of $\de$.
	If there exists the corresponding decomposition
	$L^2(\mu,H_0)=X_1+X_2$, as in ~\eqref{decomp-X},
	then $(W x)|\de\sm U_j=0$ for $x\in X_j$.
	By quasianalyticity, $(W x)|\de=0$ for all $x\in X_j$ and
	hence for all $x\in L^2(\mu,H_0)$, which is obviously false.
	This shows that $T$ cannot be decomposable.
	
\end{example}

\begin{rem}
	Notice also that Corollary~\ref{cor-decomp} ensures decomposability in many cases
	when the zero set of $\psi$ is larger than it is allowed by
	Theorem~\ref{corr-thm-direct-sums} (which assumes it to be of zero length).
	The arguments similar to those in the above example
	show that, whenever the zero set of $\psi$ on some set
	$\C\sm\cG_s(\mu)$ is connected, operator $T$ will not be
	dissectible, in general.
	
	However, Corollary~\ref{cor-decomp} says nothing if the zero set
	of $\psi$ contains an open set $U$. If $T$ is a rank one perturbation
	of $N$, then by applying the model Theorem~\ref{thm-model} and the
	diagonalization map $w\in \cQ(\cB)\mapsto w|U$ (which is well-defined, because $\Psi|U=\psi|U=0$),
	it is easy to see that $T$ is not decomposable. We do not know whether this
	fact extends to higher rank perturbations.
\end{rem}

Notice that condition~\eqref{eq-cond-domar} resembles
the well-known decomposability criterion due to
Lyubich and Macaev \cite{LM}.

Diagonalization operators as above are certainly important in the spectral study
of perturbations of normal operators and have been used intensively in~\cite{Y1}.
Some additional material can be found in unpublished preprint~\cite{Y2} by the second author,
where completeness of generalized eigenvectors of $T$ of several kinds
has been discussed. Later, the local spectral multiplicity and
completeness of ``systems of generalized eigenvectors'' have been defined
and studied in~\cite{Y3} for a general Banach space operator.

We turn now to the proofs.
We return to the functional model described in the previous sections.
Henceforth we adopt the notation
$
W_0 x= W_{0, \wh \C}\, x,
$
where $W_{0, \wh \C}$ is the transform defined in \eqref{def-W-0-Om} for the case of $\Om=\wh \C$.
We set
\[
\Mod(\wh\C)=W_0 L^2(\mu, H_0).
\]
Notice that for any $x$ in $\LtwomuH$,
\[
\langle x, v\rangle = \big(zW_0 x(z)\big)\big|_{z=\infty} .
\]
Moreover, by ~\eqref{intertw-W0-new},
\beqn
\label{intertw-W0}
W_0 T x(z) = z  W_0 T x(z) + \Psi(z)\,\langle x, v\rangle, \quad x\in \LtwomuH.
\neqn

We will use the space $\cO(\cD, \Mod(\wh\C))=\cO(\cD) \otimes \Mod(\wh\C)$.
For a function $f(\la,z)$ in this space,
$f(\la,\cdot)$ is an element of $\Mod(\wh\C)$ for any $\la\in\cD$ and
the map $\la\mapsto f(\la,\cdot)\in \Mod(\wh\C)$ is analytic.
We start by the following observation.

\begin{lemma}
	\label{lem-restr-to-diag}
	Let $\cD$ be any domain in $\C$ and let $\cD_0$ be its subdomain with admissible boundary
	such that $\overline{\cD}_0\subset\cD$. Then the restriction-to-diagonal map
	\[
	f(\la,z)\in \cO(\cD, \Mod(\wh\C))
	\mapsto f(z,z)\in \Mod(\cD_0)
	\]
	is well defined and continuous. It sends to $0$ any function of the form
	$(z-\la)f(\la,z)$, where $f(\la,z) \in \cO(\cD, \Mod(\wh\C))$.
\end{lemma}

\begin{proof}
	Let $\cD_1$ be a domain with smooth boundary, relatively compact in $\cD$ and
	containing the closure of $\cD_0$. Due to the nuclearity of $\cO(U)$, for any open set $U$,
	the restriction map factors
	through the complete projective tensor product
	\[
	\cO(\cD) \otimes {\rm Mod}(\wh\C)  \longrightarrow L^2_a(\cD_1)
	\hat{\otimes}_\pi {\rm Mod}(\wh\C)  \longrightarrow \cO(\cD_0)\otimes {\rm Mod}(\wh\C).
	\]
	Above $L^2_a(\cD_1)$ stands for the Bergman space.
	
	Choose an orthonormal basis
	$\{f_n(\la)\}$ of $L^2_a(\cD_1)$ such that any $f\in \cO(\cD)$ expands when restricted to $\cD_1$ into a
	convergent series $\sum c_n f_n(\la)$.
	Thus, after taking restrictions to $\cD_1$ of the elements of $\cO(\cD) \otimes {\rm Mod}(\wh\C)$ we can work with
	convergent series with respect to the projective norm:
	\beqn
	f(\la,z)=\sum f_n(\la) g_n(z)
	\neqn
	where $g_n(z)\in \Mod(\wh\C)$.
	Moreover, the maps $f\in L^2_a (\cD_1, \Mod(\wh\C))\mapsto g_n|\cD_0$
	are bounded and their norms decay exponentially as $n\to\infty$. This implies
	that the above restriction-to-diagonal operator is bounded. If
	$f(\la,z)$ is expressed in a series as above, then
	the restriction-to-diagonal operator sends both functions
	$zf(\la,z)$ and $\la f(\la,z)$ to $\sum z f_n(z) g_n(z)$, which implies the last statement.
\end{proof}

\begin{proof}[Proof of Theorem~\ref{thm-beta-decomp}]
	We have to prove that for every domain $\cD$ in $\C$
	and every sequence of $\LtwomuH$-valued analytic functions $x_n(\la)$ defined for $\la \in \cD$, subject to
	\beqn
	\label{beta-xn-T}
	\lim_n  (T-\la) x_n(\la) = 0
	\neqn
	with respect to the topology of $\cO(\cD, \LtwomuH)$ satisfies
	\[
	\lim_n  x_n = 0 \quad \text{in $\cO(\cD, \LtwomuH))$.}
	\]
	Notice that \eqref{beta-xn-T} rewrites as
	\beqn
	\label{beta-xn-N}
	\lim_n \big[ (N-\la) x_n(\la) + u^t \langle x_n(\la), v\rangle \big]= 0 \quad \text{in }\cO(\cD, \LtwomuH).
	\neqn
	We already know that normal operators (and more general, all decomposable operators)
	have property $(\beta)$. Hence it is sufficient to prove that
	\beqn
	\label{goes-to-zero}
	\lim_n  \langle x_n, v \rangle = 0  \quad \text{in } \cO(\cD, \C^m).
	\neqn
	
	Set $f_n(\la, \cdot)=W_0 x_n(\la)$.
	
	By applying the isomorphism $W_0$ to~\eqref{beta-xn-T}
	and using~\eqref{intertw-W0}, one finds
	\[
	(z-\la) f_n(z,\la) + \Psi(z) \langle x_n(\la), v\rangle \to 0 \quad \text{in }\cO(\cD,\Mod(\wh\C)).
	\]
	
	Put $a_n(\la)=\langle x_n(\la), v\rangle$.
	Take any compact subset $K$ of $\cD$.
	There exists a domain $\cD_0$ with admissible boundary
	such that $K \subset \cD_0\subset \overline{\cD}_0\subset \cD$.
	By virtue of the above Lemma,
	\[
	\Psi(z) a_n(z)\to 0 \quad\text{in }\Mod(\cD_0).
	\]
	Then
	\[
	\max_{z\in K}\|a_n(z)\|
	\le
	C(\cD_0,K,\tilde\psi) \|\tilde \psi a_n\|_{L^1(\cD_0)}
	\le
	C(\cD_0,K, \tilde\psi)\| \Psi a_n\|_{L^1(\cD_0)} \to  0
	\]
	as $n\to\infty$.
	This implies $a_n\to 0$ in $C(K)\otimes \C^m$. We conclude that relation ~\eqref{goes-to-zero} holds.
\end{proof}

\begin{proof}[Proof of Theorem~\ref{thm-beta-decomp-domar}]
	We prove that the function $\tilde\psi$ has no deep zeros.
	First, observe that for any two $m\times m$ invertible matrices $A$ and $B$,
	formula $B^{-1}-A^{-1}=A^{-1}(A-B)B^{-1}$ implies
	\[
	\big|\|A^{-1}\|^{-1}-\|B^{-1}\|^{-1}\big|\le \|A-B\|.
	\]
	By passing to a limit, we extend the validity of this inequality to arbitrary
	$m\times m$ matrices $A$ and $B$. This implies $\tilde\psi$ is H\"older-$\alpha$ whenever
	the matrix-valued function $\Psi$ has the same H\"older exponent.
	
	Assume $|\tilde\psi(z)-\tilde\psi(w)|\le C_0 |z-w|^\alpha$ for all $z,w\in \C$.
	Fix a bounded domain $\cD$ and its compact subset $K$ and take an arbitrary
	function $f$, analytic in $\cD$;  we have to check~\eqref{eq-deep-zeros}.
	Since $|\tilde \psi|>\de>0$ on the complement of $B(0,R)$, we may assume that
	$\cD\subset B(0,R)$. Choose a compact set $K_0$ such that $K\subset K_0^o\subset K_0\subset \cD$.
	The distance $d_0:=\dist(K_0, \pt\cD)$ is positive. Put
	\[
	r_0(z)=\min\big(d_0, \, (2C_0)^{-1/\alpha}|\tpsi(z)|^{1/\alpha}\big).
	\]
	Let $z\in K_0$. Then any point $w\in B(z,r_0(z))$ belongs to $\cD$ and
	\[
	|\tpsi(z)-\tpsi(w)|\le C_0|z-w|^\alpha \le |\tpsi(z)|/2.
	\]
	Hence
	$|\tpsi(w)|\ge|\tpsi(z)|/2$ for all $w\in B(z,r_0(z))$. Consequently
	\begin{multline*}
	|f(z)|\le \frac 1 {\pi r_0(z)^2}\int_{B(z,r_0(z))}|f|\, dA \\
	\le
	\frac {2} {|\tpsi(z)|\,\pi r_0(z)^2}\int_{B(z,r_0(z))}|\tpsi(w) f(w)|\, dA(w)
	\le
	C |\tpsi(z)|^{-1-2/\alpha}\,\int_{\cD}|\tpsi f|\, dA.
	\end{multline*}
	Hence the subharmonic function $\log |f|$ satisfies $\log |f|\le C' \log(e+|\tpsi|^{-1})$ on $K_0$.
	By applying Theorem 1 of \cite{Do} (with $\al=n=2$), we deduce that
	condition~\eqref{eq-cond-domar} ensures the local boundedness of $f$.
\end{proof}

If
(D${}'$) holds and the measure
$\mu$ satisfies an estimate $\mu(Q)\le C|Q|^\si$ for all squares $Q$,
where $\si\in (1, 2]$ and $C$ are constants, then
$\Psi$ is a H\"older-$\al$ function for some positive $\al$.
This is shown along the lines of the last part of the proof of
Lemma~\ref{L-infty-lemma}, using the estimate~\eqref{est} with some minor modifications.
We leave the details to the reader.

\begin{rem}
	As mentioned above, the case when the ``dimension'' of $\mu$ is one
	is the most difficult for our approach. However,
	if the spectrum of $N$ lies on a smooth curve and $T-N$ is compact, belongs to the Matsaev class and
	the spectrum of $T$ does not ``fill in'' the interior of the curve, then, by a
	result by Radjabalipour and Radjavi~\cite{Radjabalipour-Radjavi},
	$T$ is decomposable. This fails for larger compact perturbations,
	see~\cite{Herrero}. We also refer to~\cite{Miller-Miller-Neumann} for a study of the property
	$(\be)$ and decomposability of unilateral and bilateral weighted shifts; the later
	can be viewed as perturbations of the unweighted bilateral shift, which is unitary.
	We also can mention that functional models for perturbations of normal operators
	with spectrum on a straight line or on a curve have been devised
	by Naboko~\cite{Naboko} and by Tikhonov~\cite{Tikh}.
\end{rem}

\section{FINAL REMARKS}

The present article leaves open a few strings.

\begin{question}
	Is it possible to give sufficient conditions for dissectability of $T$ along
	a curve $\ga$, which can intersect the spectrum of $T$ in a set of positive length,
	which would be applicable to an arbitrary measure $\mu$, at least for
	sufficiently smooth perturbations, in a sense to be made precise?
\end{question}

Our construction does not apply to non-dissectible measures $\mu$. Notice, however, that a
simplest and most representative measure $\mu$ of this type is one that
is absolutely continuous with respect to $\cH^1$, restricted to a smooth curve. For this case,
any finite rank perturbation of $N$ is decomposable, due to the above-cited
paper~\cite{Radjabalipour-Radjavi}. This is, in fact, much better behavior than in our case, because no
conditions on the finite rank perturbation are necessary.
So we ask whether it is possible to merge these two cases and
to design a technique which would work for arbitrary measures.

\begin{question}
	Let $N$ be a normal operator and let $\mu$ be its
	scalar spectral measure.
	For which measures $\mu$, can one prove that \emph{any} finite rank perturbation
	of $N$ has a nontrivial invariant subspace?
\end{question}

\begin{question}
	If $T$ is a finite rank perturbation of a normal operator,
	is it true that $T$ has property $(\beta)$ if and only if
	this is true for $T^*$? Is it true for a
	suitable subclass of finite rank perturbations of normals?
\end{question}

\begin{question}
	Give necessary conditions and sufficient conditions for the perturbation $T$ to be
	\emph{similar} to a normal operator.
\end{question}

Notice that Theorems \ref{thm-beta-decomp} and \ref{thm-beta-decomp-domar} imply a sort
of necessary condition (any operator similar to normal is decomposable). However,
as examples show, what we obtain is very far from optimal.

For the case studied in \cite{Y1}, it has been proved there that $T$ is similar to
$N$ if and only if $\psi\ne 0$ everywhere on $\C$. However, this cannot be true, for instance,
if $\mu$ is a discrete measure.

\begin{question}
	Given an operator $N$, the corresponding measure $\mu$ and one more measure $\nu$ on the plane,
	when is it possible to find a finite rank perturbation $T$ of $N$, similar to a normal operator $N_1$,
	whose scalar spectral measure
	is $\nu$?
\end{question}

We refer to \cite{Poltor} for answers in the case of rank one perturbations and selfadjoint
operators $N$ and $N_1$. Notice that compact perturbations of normal operators
that are normal themselves is a particular case of Voiculescu's study \cite{Vo1} of
the scattering theory for commuting tuples of selfadjoint operators.

%
%
%

\end{document}